\documentclass[a4paper,10pt]{article}%

\usepackage{amssymb}
\usepackage{amsfonts}
\usepackage{amsmath}
\usepackage{graphicx}
\usepackage{enumerate}
\usepackage[all]{xy}
\usepackage{pstricks-add}

\newtheorem{theorem}{Theorem}

\newtheorem{claim}{Claim}

\newtheorem{corollary}[theorem]{Corollary}

\newtheorem{definition}[theorem]{Definition}

\newtheorem{lemma}[theorem]{Lemma}

\newtheorem{proposition}[theorem]{Proposition}
\newtheorem{remark}[theorem]{Remark}

\newenvironment{proof}[1][Proof]{\noindent\textbf{#1.} }{\hfill \rule{0.5em}{0.5em}}

\newcommand{\Ocal}{\mathcal O}

\begin{document}

\title{On the degree-$1$ Abel map for nodal curves}
\author{Aldi Nestor de Souza\\{\scriptsize aldinestor@cpd.ufmt.br}
\and Frederico Sercio\thanks{Supported by
CNPq, Proc. 142165/2010-7.}\\{\scriptsize fred.feitosa@ufjf.edu.br}}
\maketitle

\begin{abstract}
\noindent Let $C$ be a nodal curve and $L$ be an invertible sheaf on $C$. Let
$\alpha_{L}:C\dashrightarrow J_{C}$ be the degree-$1$ rational Abel map, which
takes a smooth point $Q\in C$ to $\left[  m_{Q}\otimes L\right]  $ in the
Jacobian of $C$. In this work we extend $\alpha_{L}$ to a morphism
$\overline{\alpha}_{L}:C\rightarrow\overline{J}_{E}^{P}$ taking values on
Esteves' compactified Jacobian for any given polarization $E$. The maps
$\overline{\alpha}_{L}$ are limits of Abel maps of smooth curves of the type
$\alpha_{L}$.
\end{abstract}

\bigskip


\bigskip

\textbf{Mathematics Subject Classification (MSC).} 14H50


\section{Introduction}

Let $C$ be a nodal curve defined over an algebraically closed field $k$. Let $J_{C}$ be the Jacobian of $C$ and $L$ be an invertible sheaf on $C$. The aim
of this article is to construct a resolution of the rational Abel map%
\[
\alpha_{L}:%
\begin{array}
[t]{ccl}%
C & \dashrightarrow & J_{C}\\
Q & \mapsto & \left[  m_{Q}\otimes L\right]  ,
\end{array}
\]
where $m_{Q}$ is the ideal sheaf of the point $Q$. If $Q\in C$ is smooth, then
$\alpha_{L}(Q)$ is well-defined but if $Q\in C$ is a node, then $\alpha
_{L}(Q)$ is not defined because $m_{Q}$ is noninvertible.

In order to solve that map, that is, to extend the Abel map on the whole $C$%
, we must enlarge the target space of $\alpha_{L}$, which leads us to the problem of how to find a good compactification for the Jacobian. The problem of the compactification goes back at least to Igusa \cite{I}. Later on D'Souza, following a suggestion of Mumford and Mayer, obtained in his thesis \cite{DS} a compactification of relative Jacobian of a family of irreducible curves with nodes and cusps as singularities under somewhat restrictive hypothesis. One year later, Altman and Kleiman \cite{AK80} gave a good solution for the case of families of geometrically integral curves. Their relative compactification parametrizes torsion-free and rank-$1$ sheaves on the fibers, and it admits a universal sheaf after an \'{e}tale base change.

For reducible nodal curves Oda-Seshadri \cite{OS} and Seshadri \cite{S} produced some compactifications. But these compactifications are not applicable to families of reduced curves. In her thesis \cite{C}, Caporaso constructed a compactification for the relative generalized Jacobians of families of stable curves. One year later, Pandharipande \cite{PP} produced in his thesis an equivalent construction, valid for higher ranks as well. At nearly the same time, Simpson \cite{Si} constructed moduli spaces of coherent sheaves over any family of projective varieties. The compactifications by Caporaso, Pandharipande and Simpson are not fine moduli spaces, and thus do not carry a Poincar\'{e} sheaf.

At last, Esteves considered in \cite{E01} the algebraic space constructed by Altman and Kleiman \cite{AK80}, parametrizing torsion-free rank-$1$ simple sheaves on the fibers of a family of curves and he showed that this space is universally closed over the base, and consequently one can regard it as a compactification of the relative Jacobian. This compactification is a fine moduli space, and hence it does admit a Poincar\'{e} sheaf after an \'{e}tale base change. In this work we consider Abel maps into Esteves' compactification.

We recall that the map $\alpha_{L}$ above is the case $d=1$ of the more
general rational degree-$d$ Abel map $\alpha_{L}^{d}$ defined, for a positive
integer $d$ and a line bundle $L$, in the following way:%
\[
\alpha_{L}^{d}:%
\begin{array}
[t]{ccl}%
C^{d} & \dashrightarrow & J_{C}\\
(Q_{1},\cdots,Q_{d}) & \mapsto & \left[  m_{Q_{1}}\otimes\cdots\otimes
m_{Q_{d}}\otimes L\right]  \text{.}%
\end{array}
\]
When $C$ is smooth, $\alpha_{L}^{d}$ is a morphism and a well-know result of
Abel says that the fibers of Abel map of degree $d$ are projectivized complete
linear series (up to the natural action of the $d$-th symmetric group). So,
when $C$ is smooth, all the possible embeddings of $C$ in projective spaces
are known once we know its Abel maps.

Some particular cases of the Abel maps have been solved. Altman and Kleiman in
\cite{AK80} considered the problem for integral curves. For reducible curves,
the problem of completing the Abel maps is open with a few exceptions:
Caporaso and Esteves in \cite{CE} constructed degree-$1$ Abel maps for stable
curves; Caporaso, Coelho and Esteves in \cite{CCoE} constructed a degree-$1$ Abel maps for Gorenstein curves; Coelho and Pacini in \cite{CoP} constructed
Abel maps of any degree for curves of compact type; Pacini in \cite{P1} and \cite{P2} constructed a degree-$2$ Abel maps for nodal curves and finally Abreu, Coelho and Pacini in \cite{ACoP} constructed degree-$d$ Abel maps for nodal curves with two components. In all cases the authors have been used a specific polarization and a particular $L$.

The general strategy to solve the Abel maps is to resort to families of
curves. More precisely, let $C$ be a nodal curve of genus $g$ and 
$f:\mathcal{C}\rightarrow B$ be a regular local smoothing of $C$, i.e., a
family of curves where $\mathcal{C}$ is smooth and $B$ is the spectrum
of a DVR (discrete valuation ring) with residue field $k$ and quotient field
$K$, and such that $f$ has special fiber isomorphic to $C$ and smooth generic
fiber $\mathcal{C}_{K}$. Let $\sigma:B\rightarrow\mathcal{C}$ be a section of
$f$ through the $B$-smooth locus of $\mathcal{C}$. Let $\mathcal{E}$ be a polarization on $\mathcal{C}$, i. e., a vector bundle on $\mathcal{C}$ such that rk($\mathcal{E}$) divides deg($\mathcal{E}$). Let $L$ be a line bundle on $C$ of degree $g-\mu(\mathcal{E})$, and consider a deformation $\mathcal{L}$ of $L$, i. e., an invertible sheaf $\mathcal{L}$ on $\mathcal{C}$ such that $\left.\mathcal{L}\right\vert _C=L$.

Consider the scheme $\overline{J}_{\mathcal{E}}^{\sigma}$ 
constructed in \cite{E01}, parametrizing torsion-free rank-$1$ sheaves $I$ of degree ($g-1-\mu(\mathcal{E})$) on $\mathcal{C}/B$ that are $\sigma$-quasistable with respect to $\mathcal{E}$. This means that $I$ satisfies certain numerical conditions depending on the degree of $I$ in each component of $C$. We recall that $\overline
{J}_{\mathcal{E}}^{\sigma}$ is a proper $B$-scheme.

Given $\mathcal{L}$ and $\mathcal{E}$ as above, we have a rational map
\[
\alpha_{\mathcal{L},\mathcal{E}}:%
\begin{array}
[t]{ccl}%
\mathcal{C} & \dashrightarrow & \overline{J}_{\mathcal{E}}^{\sigma}%
\end{array}
\]
defined over $\mathcal{C}_K$ which takes $Q\in\mathcal{C}_{K}$ to $\alpha_{\mathcal{L},\mathcal{E}}(Q)=\left[
m_{Q}\otimes \left.\mathcal{L}\right\vert _{\mathcal{C}_K}\right]  $.
Our aim is to extend this map to the whole $\mathcal{C}$.
Since $\overline{J}_{\mathcal{E}}^{\sigma}$ is a fine moduli space, to extend
the map $\alpha_{\mathcal{L},\mathcal{E}}$ to the whole $\mathcal{C}$ it is enough to
give a relatively torsion-free rank-$1$ $\sigma$-quasistable sheaf $\mathcal{M}$ on the family 
\[
p_1:\mathcal{C}\times_{B}\mathcal{C}\rightarrow \mathcal{C}
\]
given by the projection map $p_1$ onto the first factor. The moduli map induced by $\mathcal{M}$ is given by its restriction over the fibers of the family $p_1$.

As we will see in Theorem \ref{teo:main}, for any invertible sheaf $\mathcal{L}$ and for any polarization $\mathcal{E}$ we can show that $\alpha_{\mathcal{L},\mathcal{E}}:\mathcal{C} \dashrightarrow \overline{J}_{\mathcal{E}}^{\sigma}$
is already a morphism. More precisely, let $\phi: \widetilde{\mathcal{C}}^2\rightarrow \mathcal{C}^2$ be a good partial desingularization. We will be able to construct an invertible sheaf $\widetilde{\mathcal{F}}$ over $\widetilde{\mathcal{C}}^2$, such that $\phi_*\widetilde{\mathcal{F}}$ is a relatively torsion-free, rank-$1$, $\widetilde{\sigma}$-quasistable sheaf over $\mathcal{C}^2$, where $\widetilde{\sigma}:\mathcal{C}\rightarrow \mathcal{C}^2$ is the section of the projection $\mathcal{C}^2\rightarrow\mathcal{C}$ onto the second factor induced by the section $\sigma:B\rightarrow \mathcal{C}$. This sheaf induces a morphism  
\[
\overline{\alpha}_{\mathcal{L},\mathcal{E}}:\mathcal{C}\rightarrow\overline{J}_{\mathcal{E}}^{\sigma}
\] 
which takes $Q \in \mathcal{C}$ to 
\[
\left. \phi_*\widetilde{\mathcal{F}} \right\vert _{p_1^{-1}(Q)}
\]
restricting to $\alpha_{\mathcal{L},\mathcal{E}}$ over the smooth locus of $f:\mathcal{C}\rightarrow B$.

We recall that in \cite{CCoE}, \cite{CE} and \cite{Co} the authors have been
used a particular polarization and the particular $L=\mathcal{O}_{C}(P)$. Notice that we use the approach used by Caporaso and Esteves in \cite{CE} where the obstruction to extend the Abel map is overcome by using a special type of invertible sheaves, named twisters. Rocha, in his thesis \cite{R} constructed degree-$1$ and degree-$0$ Abel maps avoiding the use of twisters, putting Simpson's compactified Jacobians as the target of Abel maps.

In short, here is a summary of this article. In Section 2 we review the
technical background, in particular the concepts of
$P$-quasistability and $\sigma$-quasistability. In Sections 3 and 4 we define
twisters and twister difference. In Section 5 we construct the sheaf
$\mathcal{M}$ on $\mathcal{C}\times_{B}\mathcal{C}/\mathcal{C}$ which solves
the first Abel map.

\section{Technical background}

Let $k$ be an algebraically closed field. A \textit{curve} is a connected,
projective and reduced scheme of dimension $1$ over $K$. A \textit{subcurve} $Y$ of a curve $C$ is a reduced subscheme of pure dimension $1$, or equivalently, a reduced
union of irreducible components of $C$. A \textit{node} of a curve $C$ is a singular point $N$ of $C$ such that $\widehat{\Ocal}_{C,N}=k[[x,y]]/(xy)$. A node $N$ of $C$ is called a \textit{separating node} if there is a subcurve $Y$ of $C$ such that $Y\cap Y'=\{N\}$, where $Y'=\overline{C\setminus Y}$. A node $N$ of $C$ is called \textit{reducible} (or external) if there are $C_i$ and $C_j$, distinct irreducible components of $C$, such that $N \in C_i \cap C_j$, otherwise, $N$ is called an \textit{irreducible} (or internal) node. A \textit{nodal curve} is a curve with only nodal singularities.  We denote by $C^{\text{sing}}$ and $C^{\text{sm}}$ respectively the singular and smooth locus of $C$. 

If $Y,Z\subseteq C$ are subcurves, we
let $Y\wedge Z$ denote the maximum subcurve of $C$ contained in $Y\cap Z$; we
let $\overline{Z-Y}$ denote the minimum subcurve containing $Z\setminus Y$. If
$Y,Z\subseteq C$ are subcurves such that $\dim\left(  Y\cap Z\right)  =0$,
define $\delta_{YZ}$ as the number of nodes in $Y\cap Z$. In addition, if
$Y=C_{i}$ and $Z=C_{j}$ are irreducible components of a curve $C$ we use
$\delta_{i,j}$ to denote $\delta_{C_{i}C_{j}}$.

A \textit{chain of rational curves} is a curve whose components are smooth and rational and can be ordered, $E_1,...,E_n$, in such a way that $\#(E_i \cap E_{i+1})=1$ for $i=1,...,n-1$ and $E_i \cap E_j = \emptyset$ if $|i-j| > 1$. If $n$ is the number of components, we say that the chain has \textit{length} $n$. The components $E_1$ and $E_n$ are called the \textit{extreme curves} of the chain. 

Let $\mathcal{N}$ be a collection of nodes of $C$, and $\eta : \mathcal{N} \rightarrow \mathbb{N}$ a function. Denote by $\widetilde{C}_{\mathcal{N}}$ the partial normalization of $C$ at $\mathcal{N}$. For each $P \in \mathcal{N}$, let $E_P$ be a chain of rational curves of length $\eta(P)$. Let $C_{\eta}$ denote the curve obtained as the union of $\widetilde{C}_{\mathcal{N}}$ and the $E_P$ for $P \in \mathcal{N}$ in the following way: each chain $E_P$ intersects no other chain, but intersects $\widetilde{C}_{\mathcal{N}}$ transversally at two points, the branches over $P$ on $\widetilde{C}_{\mathcal{N}}$ on one hand, and nonsingular points on each of the two extreme curves of $E_P$ on the other hand. There is a natural map $\mu_{\eta} : C_{\eta} \rightarrow C$ collapsing each chain $E_P$ to a point, whose restriction to $\widetilde{C}_{\mathcal{N}}$ is the partial normalization map. The curve $C_{\eta}$ and the map $\mu_{\eta}$ are well-defined up to $C$-isomorphism. A rational curve in any of the chain $E_P$ is called $\mu_{\eta}$\textit{-exceptional curve}. We call the curve $C_{\eta}$ (or the map $\mu_{\eta}$), a \textit{semistable modification of} $C$. If $\eta(P)=1$ for each $P \in \mathcal{N}$, then $C_{\eta}$ (or the map $\mu_{\eta}$) is called a \textit{quasistable modification of} $C$.

There are two special cases of the above construction that will be useful for us. First, if $\mathcal{N}=\{R\}$ and $\eta$ takes $R$ to $1$, let $C_R := C_{\eta}$ and $\mu_R := \mu_{\eta}$. Second, if $\mathcal{N}=\mathcal{N}(C)$, where $\mathcal{N}(C)$ is the collection of reducible nodes of $C$, and $\mu : \mathcal{N}(C) \rightarrow \mathbb{N}$ is the constant function with value $m$, let $C(m) := C_{\eta}$ and $\mu(m) := \mu_{\eta}$. Set $C(0) := C$ and $\mu(0) := \text{id}_C$.

Given a map of curves $\phi:C^{\prime}\rightarrow C$, we say that an
irreducible component of $C^{\prime}$ is $\phi$-\textit{exceptional} if it is
a smooth rational curve and is contracted by the map.

Let $I$ be a coherent sheaf on a curve $C$. We say that $I$ is
\textit{torsion-free} if its associated points are generic points of $C$. We
say that $I$ is \textit{of rank }$1$ if $I$ is invertible on a dense open
subset of $C$. If $I$ is a rank-$1$ torsion-free sheaf, we call $\deg
(I):=\chi(I)-\chi\left(  \mathcal{O}_{C}\right)  $ the \textit{degree} of $I$
and we define 
\[
I_{Y}:=\frac{\left.  I\right\vert _{Y}}{\mathcal{T}(\left.
I\right\vert _{Y})},
\]
where $\mathcal{T}\left(  \left.  I\right\vert
_{Y}\right)  $ is the torsion subsheaf of $\left.  I\right\vert _{Y}$. Note
that $I_{Y}$ is torsion-free. A sheaf $I$ is said to be \textit{simple} if
End$(I)=k$, or equivalently, if $I$ is not decomposable \cite[Lemma 1.1.5, p.
11]{Co}.

Let $E$ be a vector bundle on a curve $C$. The \textit{slope} of $E$ is the
number 
\[
\mu(E):=\frac{\deg(E)}{\text{rk}(E)}.
\]
A \textit{polarization} on a
curve $C$, in the sense of Esteves \cite{E01}, is a vector bundle $E$ on $C $
whose slope is an integer, that is, such that rk$(E)$ divides $\deg(E)$. A
torsion-free rank-$1$ sheaf $I$ on a curve $C$ is \textit{semistable} with
respect to $E$ if $\chi(E\otimes I)=0$ and $\chi(E\otimes I_{Y})\geq0 $, for
all proper subcurves $Y$ of $C$. If $P$ is a nonsingular point of $C $, we say
that a torsion-free rank-$1$ sheaf $I$ is $P$-\textit{quasistable} with
respect to $E$ if $I$ is semistable and in addition $\chi(E\otimes I_{Y})>0$
for every proper subcurve $Y$ of $C$ containing $P$. Let%
\[
\beta_{Y}(I):=\chi(I_{Y})+\frac{\deg_{Y}(E)}{\text{rk}(E)}\text{.}%
\]
As $\chi(E\otimes I_{Y})=\,$rk$(E)\chi(I_{Y})+\deg_{Y}(E)$, the sheaf $I$ is
$P$-quasistable if and only if $\beta_{Y}(I)\geq0$ for every proper subcurve
$Y$ of $X$ and $\beta_{Y}(I)>0$ for every subcurve $Y$ such that $P\in Y$.

A \textit{family of (connected) curves} is a proper and flat morphism $f:\mathcal{C}%
\rightarrow B$ whose geometric fibers are connected curves. If $b\in B$, we denote by
$\mathcal{C}_{b}:=f^{-1}(b)$ its fiber. The family $f:\mathcal{C}\rightarrow
B$ is called \textit{local} if $B=\,$Spec$\left(  K\left[  \left[  t\right]
\right]  \right)  $, \textit{regular} if $\mathcal{C}$ is regular and
\textit{pointed} if it is endowed with a section $\sigma:B\rightarrow
\mathcal{C}$ through the smooth locus of $f$. A \textit{smoothing} of a curve
$C$ is a regular local family $f:\mathcal{C}\rightarrow B$ with special fiber
$C$. A \textit{sheaf} on $\mathcal{C}/B$ is a $B$-flat coherent sheaf on $\mathcal{C}$. Given a pointed smoothing $f:\mathcal{C}\rightarrow B$ of a curve $C$
with section $\sigma:B\rightarrow\mathcal{C}$, we define $P:=\sigma(0)$. If
$f:\mathcal{C}\rightarrow B$ is a family of curves, we denote by
$\mathcal{C}^{2}$ the product $\mathcal{C}\times_{B}\mathcal{C}$ and by
$\mathcal{C}_{T}$ the product $\mathcal{C}\times_{B}T$ where $T$ is a $B$-scheme.

Let $f:\mathcal{C} \rightarrow B$ be a family of curves and let $\mathcal{I}$ be a sheaf on $\mathcal{C}/B$. The sheaf $\mathcal{I}$ is called \textit{torsion-free} (resp. \textit{rank}-$1$ and \textit{simple}) on $\mathcal{C}/B$ if, for each $b \in B$, the restriction $\mathcal{I}(b)$ is torsion-free (resp. rank-$1$ and simple).

Let $f:\mathcal{C}\rightarrow B$ be a smoothing of $C$. The \textit{relative
compactified Jacobian functor} of the family $\mathcal{C}/B$ is the
contravariant functor%
\[
\overline{\mathcal{J}}_{\mathcal{C}/B}:\left(  \text{Sch/}B\right)  ^{\circ
}\longrightarrow\left(  \text{Sets}\right)
\]
that associates to each $B$-scheme $T$ the set of simple, torsion-free,
rank-$1$ sheaves on $\mathcal{C}_{T}/T$ modulo equivalence, where we say that
two sheaves $\mathcal{F}_{1}$ and $\mathcal{F}_{2}$ on $\mathcal{C}_{T}/T$ are
equivalent if there exists an invertible sheaf $\mathcal{G}$ on $T$ such that
$\mathcal{F}_{1}\simeq\mathcal{F}_{2}\otimes p_{2}^{\ast}\left(
\mathcal{G}\right)  $, with $p_{2}:\mathcal{C}_{T}\rightarrow T$ being the
projection onto the second factor.

Since $f:\mathcal{C}\rightarrow B$ is a family of curves, the functor
$\overline{\mathcal{J}}_{\mathcal{C}/B}$ is represented by an
\textit{algebraic space} $\overline{J}_{\mathcal{C}/B}$ \cite[Theorem 7.4, p.
99]{AK80}. Esteves showed that, after a suitable \'{e}tale base change to
obtain enough sections \cite[Lemma 18, p. 3061]{E01}, $\overline
{J}_{\mathcal{C}/B}$ becomes a \textit{scheme }\cite[Theorem B, p. 3048]{E01},
consequently it is a fine moduli space. But the space $\overline
{J}_{\mathcal{C}/B}$ is not proper.

Let $f:\mathcal{C}\rightarrow B$ be a family of curves and consider a vector
bundle $\mathcal{E}$ on $\mathcal{C}$. The \textit{relative degree of
}$\mathcal{E}/B$ is $\deg(\mathcal{E}/B):=\deg_{\mathcal{C}(b)}(\det(\left.
\mathcal{E}\right\vert _{\mathcal{C}(b)}))$, where $\mathcal{C}(b)$ is the
fiber of $\mathcal{C}/B$ over any $b\in B$. The quotient 
\[
\mu(\mathcal{E}):=\frac{\text{deg}(\mathcal{E}/B)}{\text{rk}(\mathcal{E})}
\]
is called the \textit{slope} of $\mathcal{E}$. We say $\mathcal{E}$ is a \textit{polarization} if rk$(\mathcal{E})$ divides $\deg(\mathcal{E}/B)$ (or
equivalently $\mu(\mathcal{E})$ is an integer). Let $\sigma:B\rightarrow
\mathcal{C}$ be a section of $f$ through the smooth locus of $\mathcal{C}$. A
torsion-free rank-$1$ sheaf $\mathcal{I}$ on $\mathcal{C}/B$ is $\sigma
$-quasistable with respect to $\mathcal{E}$ if $\mathcal{I}(b)$ is $\sigma
(b)$-quasistable with respect to $\mathcal{E}(b)$ for every geometric point
$b$ of $B$. According to \cite[Lemma 1.3.5, p. 19]{Co}, if $\mathcal{I}$ is
$\sigma$-quasistable, then $\mathcal{I}$ is simple on $\mathcal{C}/B$.

Denote by $\overline{J}_{\mathcal{E}}^{\sigma}$ the subspace of $\overline
{J}_{\mathcal{C}/B}$ parametrizing the torsion-free rank-$1$ sheaves
$\mathcal{I}$ on $\mathcal{C}/B$ that are $\sigma$-quasistable with respect to
$\mathcal{E}$. Esteves showed that $\overline{J}_{\mathcal{E}}^{\sigma}$ is
proper over $B$ \cite[Theorem A, p. 3047]{E01}.

Since that $\overline{J}_{\mathcal{E}}^{\sigma}$ is a fine moduli space, to
give a morphism $\alpha:\mathcal{C}\rightarrow\overline{J}_{\mathcal{E}}^{\sigma}$
is equivalent to give a simple, torsion-free, rank-$1$ sheaf $\mathcal{M}$ on
$\mathcal{C\times}_{B}\mathcal{C}/\mathcal{C}$ which is $\sigma$-quasistable
on the fibers. Given $\mathcal{L}$ an invertible sheaf on $\mathcal{C}/B$ and
$\mathcal{E}$ a polarization on $\mathcal{C}$, we have a rational map%
\[
\alpha_{\mathcal{L}}:%
\begin{array}
[t]{ccl}%
\mathcal{C} & \dashrightarrow & \overline{J}_{\mathcal{E}}^{\sigma},
\end{array}
\]
where $\alpha_{\mathcal{L}}(Q)=\left[  m_{Q}\otimes\mathcal{L}_{K}\right]  $
for $Q\in\mathcal{C}_{K}$, $\mathcal{C}_{K}$ is the generic fiber of
$f:\mathcal{C}\rightarrow B$ and $\mathcal{L}_{K}=\left.  \mathcal{L}%
\right\vert _{\mathcal{C}_{K}}$. So, to extend this map it suffices to give a
simple, torsion-free, rank-$1$ sheaf $\mathcal{M}$ on $\mathcal{C\times}%
_{B}\mathcal{C}/\mathcal{C}$ which is $\sigma$-quasistable on the fibers, such
that for $Q\in\mathcal{C}_{K},$%
\[
\left.  \mathcal{M}\right\vert _{\mathcal{C\times}_{B}\{Q\}}=m_{Q}%
\otimes\left.  \mathcal{L}\right\vert _{\mathcal{C}_{K}}.
\]

Given a smoothing $\mathcal{C}/B$ of $C$, a \textit{twister} of $\mathcal{C}%
/B$ is a line bundle of degree-$0$ on $C$ of the form $\left.  \mathcal{O}%
_{\mathcal{C}}(Z)\right\vert _{C}$, where $Z$ is a Cartier divisor of
$\mathcal{C}$ supported in $C$, so a formal sum of components of $C$ (each component of $C$ is a Cartier divisor of $\mathcal{C}$ because $\mathcal{C}$ is regular) of the
type $\sum a_{i}C_{i}$ where $a_{i}\in\mathbb{Z}$. We set $\mathcal{O}%
_{C}(Z):=\left.  \mathcal{O}_{\mathcal{C}}(Z)\right\vert _{C}$.


\section{Nodal curves and twisters}


Let $C$ be a connected, nodal curve with components $C_{1},...,C_{p}$. Let $f:\mathcal{C}\rightarrow B$ be a regular local smoothing of $C$. Let $\sigma:B\rightarrow\mathcal{C}$ be a section of
$f$ through the $B$-smooth locus of $\mathcal{C}$. Let $\mathcal{E}$ be a polarization on $\mathcal{C}$. Let $L$ be a line bundle on $C$, and consider a deformation $\mathcal{L}$ of $L$, i. e., an invertible sheaf $\mathcal{L}$ on $\mathcal{C}/B$ such that $\left.\mathcal{L}\right\vert _C=L$.

Let $\phi:\widetilde{\mathcal{C}^{2}}\rightarrow\mathcal{C}^{2}$ be a desingularization of $\mathcal{C}^{2}$. Let $\Delta
\subset\mathcal{C}^{2}$ be the diagonal subscheme and let $\widetilde{\Delta}$
be the strict transform of $\Delta$ (via $\phi$). Denoting by $p_2:\mathcal{C}^2 \rightarrow \mathcal{C}$ the second projection, we obtain a family of curves
\begin{displaymath}
\xymatrix{
\widetilde{\mathcal{C}^{2}} \ar[r]^{\phi} & \mathcal{C}^{2} \ar[r]^{p_2} & \mathcal{C} }
\end{displaymath}

Consider the following sheaf over the family $\widetilde{\mathcal{C}^{2}}/\mathcal{C}$
\begin{equation}
\widetilde{\mathcal{M}}:=\mathcal{I}_{\widetilde{\Delta}/\widetilde{\mathcal{C}^{2}}%
}\otimes\left(  p_{2}\phi\right)  ^{\ast}\mathcal{L}\otimes\widetilde
{\mathcal{T}},\label{eqn1}%
\end{equation}
where $\widetilde{\mathcal{T}}$ is an invertible sheaf. 

We want to find a desingularization $\phi:\widetilde{\mathcal{C}^{2}}\rightarrow\mathcal{C}^{2}$ so that $\phi_*\widetilde{\mathcal{M}}$ is a relatively torsion-free, rank-$1$, $\sigma$-quasistable sheaf on $\mathcal{C}^{2}/\mathcal{C}$.

The map 
\[
\overline{\alpha}_{\mathcal{L},\mathcal{E}}:%
\begin{array}
[t]{ccl}%
\mathcal{C} & \rightarrow & \overline{J}_{\mathcal{E}}^{\sigma}
\end{array}
\]
induced by $\phi_*\widetilde{\mathcal{M}}$ coincides with $\alpha_{\mathcal{L},\mathcal{E}}$ over the generic fiber of $f$


\subsection{Good partial desingularizations}


By \cite[Section 3.1]{CoEP}, the $3$-fold $\mathcal{C}^2$ is singular exactly at the points $(R,S)$ where $R,S$ are nodes of $C$, not necessarily distinct, i. e., 
\[
\text{Sing}(\mathcal{C}^{2})=\left\{  (R,S):R,S\in
C^{\text{sing}}\right\}  \text{.}%
\]
In fact, $\mathcal{C}^2$ has a quadratic isolated singularity at $(R,S)$. This singularity can be resolved by blowing up $\mathcal{C}^2$ at $(R,S)$, at the cost of replacing the point by a quadric surface. However, for our purposes, we choose a desingularization that replaces each point $(R,S)$ by a unique smooth rational curve. Thus, firstly we need a convenient desingularization of $\mathcal{C}^2$.
For more details on the subject, we refer the reader to \cite[Sections 3 and 4]{CoEP}.

Let $\phi:\widetilde{\mathcal{C}^{2}}\rightarrow\mathcal{C}^{2}$ be a
\textit{good partial resolution of singularities} of $\mathcal{C}^{2}$ as introduced by \cite[Section 4.1, p.2936]{CoEP}, that
is, $\phi:\widetilde{\mathcal{C}^{2}}\rightarrow\mathcal{C}^{2}$ is a sequence
of blowups, starting by the blowup along the diagonal subscheme of $\mathcal{C}^{2}$
and then blowing up all strict transform of products $Y\times Z$ of irreducible components $Y$ and $Z$ of $C$ with $Y\neq Z$. According to \cite[Section 4.1, p. 2936]{CoEP}, $\widetilde{\mathcal{C}^{2}}$ is nonsingular away from the points over the
pairs $(R,S)$ of distinct nodes $R,S$ of $C$ where either $R$ or $S$ is irreducible, so the strict
transform, via $\phi$, of any product $Y\times Z$ is a Cartier divisor in
$\widetilde{\mathcal{C}^{2}}$.

In this context the family of curves
\begin{displaymath}
\xymatrix{
\widetilde{\mathcal{C}^{2}} \ar[r]^{\phi} & \mathcal{C}^{2} \ar[r]^{p_2} & \mathcal{C} }
\end{displaymath}
looks locally over a node $R \in C$ like the below diagram

\begin{figure}[!h]
\centering
\begin{pspicture}(4,0)(15,10)
\pscurve(9.5,2)(10,4)(9.5,6)
\pscurve(9.5,5)(10,7)(9.5,9)
\put(10.3,4){$C_i$}
\put(10.1,5.3){$R$}
\put(10.3,7){$C_j$}
\pscurve(4.7,2)(5.2,3.5)(4.3,5)
\pscurve(4.3,6)(5.2,7.5)(4.7,9)
\put(5.5,3.5){$\widetilde{C}_i$}
\put(4.8,5.3){$\mathbb{P}^1$}
\put(5.5,7){$\widetilde{C}_j$}
\psline[linecolor=blue](4.6,4)(4.6,6.8)
\put(4.7,9.5){$\widetilde{C}$}
\put(9.7,9.5){$C$}
\put(14,5.3){$R \in C^{\text{sing}} \subset \mathcal{C}$}
\psline{->}(6,5.5)(9,5.5)
\psline{->}(11,5.5)(13,5.5)
\end{pspicture}
\caption{Fiber over $R \in C^{\text{sing}} \subset \mathcal{C}$ of the family $p_2 \circ \phi$.}
\end{figure}
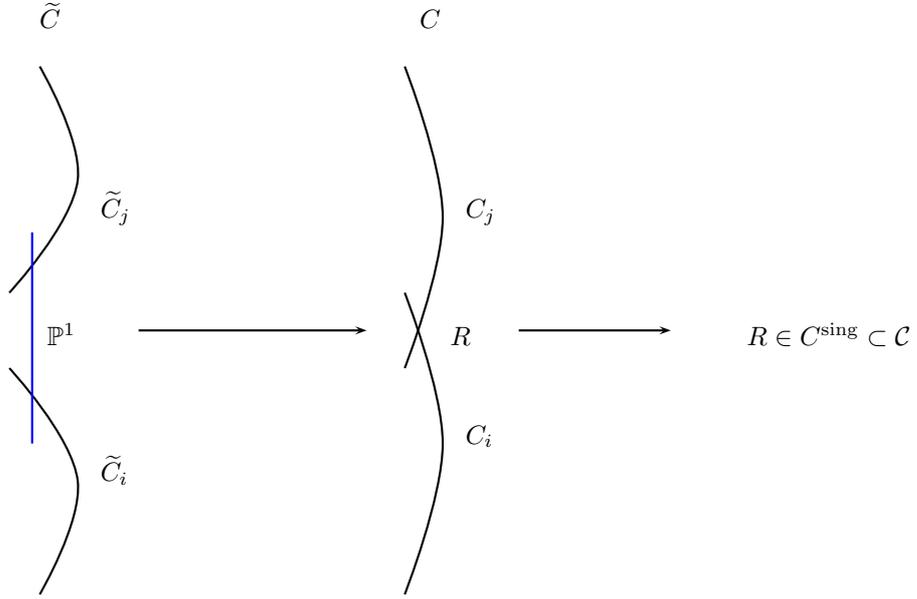

Let $C$ be a nodal curve. We define $C_R$ as the curve obtained from $C$ by replacing the node $R$ by a smooth rational curve, and $C(1)$ as the curve obtained from $C$ by replacing each reducible node of $C$ by a smooth rational curve. Let $\mathcal{I}_{\left. \Delta \right\vert \mathcal{C}^2}$ be the ideal sheaf of $\Delta \subset \mathcal{C}^2$ and let 
\[
\mathbb{P}_{\mathcal{C}^2}(\mathcal{I}_{\left. \Delta \right\vert \mathcal{C}^2}) := \text{Proj}_{\mathcal{C}^2}(\mathcal{S}(\mathcal{I}_{\left. \Delta \right\vert \mathcal{C}^2})),
\]
where $\mathcal{S}(\mathcal{I}_{\left. \Delta \right\vert \mathcal{C}^2})$ is the sheaf of symmetric algebras of $\mathcal{I}_{\left. \Delta \right\vert \mathcal{C}^2}$. 

The next three propositions summarize the properties of the good partial desingularizations we are looking for

\begin{proposition} \label{pro:diag}
Let $\phi:\widetilde{\mathcal{C}^{2}}\rightarrow\mathcal{C}^{2}$ be the blowup of $\mathcal{C}^2$ along $\Delta$. Let $\rho_i := p_i \phi$, where $p_i:\mathcal{C}^2 \rightarrow \mathcal{C}$ is the projection onto the $i$-th factor for $i=1,2$. Let $R \in C$. For $i=1,2$, let $X_i := \rho_i^{-1}(R)$ and denote by $\mu_i:X_i \rightarrow C$ the restriction of $\rho_{3-i}$ to $X_i$. Then the following statements hold:
\begin{enumerate}[1)]
\item $\widetilde{\mathcal{C}^2}$ is $\mathcal{C}^2$-isomorphic to $\mathbb{P}_{\mathcal{C}^2}(\mathcal{I}_{\left. \Delta \right\vert \mathcal{C}^2})$. 
\item $\rho_i$ is flat for $i=1,2$.
\item If $R$ is not a node of $C$, then $\mu_i$ is an isomorphism for $i=1,2$.
\item If $R$ is a node of $C$, then $X_i$ is $C$-isomorphic to $C_R$ and $\widetilde{\mathcal{C}^2}$ is regular along the rational component of $X_i$ contracted by $\mu_i$ for $i=1,2$.
\end{enumerate}
\end{proposition}
\begin{proof}
\cite[Proposition 2.2, p.2925]{CoEP}
\end{proof} \\

We have the following diagram

\begin{displaymath}
\xymatrix{
\widetilde{\mathcal{C}^{2}} \ar[r]^{\phi} \ar@/^2pc/[rr]^{\rho_2} \ar@/_/[dr]_{\rho_1} & \mathcal{C}^{2} \ar[r]^{p_2} \ar[d]^{p_1} & \mathcal{C} \ar[d]^{f} \\
& \mathcal{C} \ar[r]_{f} & B  }
\end{displaymath}
where $p_{1}$ and $p_{2}$ are, respectively, the projections onto the first
and second factors.

\begin{proposition} \label{pro:strict}
Let $R$ and $S$ be reducible nodes of $C$. Assume $R \in C_i \cap C_j$ and $S \in C_k \cap C_l$, for integers $i,j,k,l$ with $i \neq j$ and $k \neq l$. If $R=S$, assume $i=k$ and $j=l$. Let $\phi: \widetilde{\mathcal{C}}^2 \rightarrow \mathcal{C}^2$ denote the blowup of $\mathcal{C}^2$ along $C_i \times C_l$, or along the diagonal if $R=S$. Put $E := \phi^{-1}(R,S)$. Then the following statements hold:
\begin{enumerate}[1)]
\item $E$ is a smooth rational curve and $\widetilde{\mathcal{C}}^2$ is regular in a neighborhood of $E$.
\item The strict transforms of $C_i \times C_l$ and $C_j \times C_k$ contain $E$, while those of $C_i \times C_k$ and $C_j \times C_l$ intersect $E$ transversally at a unique point, distinct for each transform.
\item If $R=S$, the strict transform of the diagonal contains $E$.
\item The composition $\widetilde{\mathcal{C}}^2 \rightarrow \mathcal{C}$ of $\phi$ with the projection of $\mathcal{C}^2$ onto any of its factors is flat.
\end{enumerate}
\end{proposition}
\begin{proof}
\cite[Proposition 3.1, p.2928]{CoEP}
\end{proof} \\

The pictures \ref{fig1} and \ref{fig2} below illustrates the possible blowups along a product of irreducible subcurves of $C$.

\begin{figure}[!h]
\centering

\begin{minipage}[b]{0.45\linewidth}
\centering

\begin{pspicture}(-1.8,-1.6)(2.5,3)
\psline(-1,0)(-1,3)
\psline(0,-1)(3,-1)
\psdots(2,-1)(-1,2)
\put(1,-1.6){$C_i$}
\put(1.9,-1.5){$R$}
\put(2.5,-1.6){$C_j$}
\put(-1.8,2.5){$C_l$}
\put(-1.5,1.9){$S$}
\put(-1.8,0.9){$C_k$}
\psline(0,2)(1,2)
\psline(1,2)(1,3)
\psline[linecolor=blue](1,2)(2,1)
\psline(2,0)(2,1)
\psline(2,1)(3,1)
\psdots(1,2)(2,1)
\put(1.6,1.6){{\scriptsize $E$}}
\put(2,2.1){$\widetilde{C_j \times C_l}$}
\put(-0.2,0.8){$\widetilde{C_i \times C_k}$}
\end{pspicture}
\caption{Blow up along $C_i\times C_k$. The strict transforms of $C_i\times C_k$ and $C_j\times C_l$ contains $E$.}\label{fig1}

\end{minipage}
\hfill
\begin{minipage}[b]{0.45\linewidth}
\centering

\begin{pspicture}(0,-1.6)(4.4,3)
\psline(4,0)(4,3)
\psline(0,-1)(3,-1)
\psdots(1,-1)(4,2)
\put(0.2,-1.6){$C_i$}
\put(0.9,-1.5){$R$}
\put(1.8,-1.6){$C_j$}
\put(4.4,2.5){$C_l$}
\put(4.2,1.8){$S$}
\put(4.4,0.9){$C_k$}
\psline(0,1)(1,1)
\psline(1,0)(1,1)
\psline[linecolor=blue](1,1)(2,2)
\psline(2,2)(3,2)
\psline(2,2)(2,3)
\psdots(2,2)(1,1)
\put(1.5,1.2){{\scriptsize $E$}}
\put(1.8,0.7){$\widetilde{C_j \times C_k}$}
\put(0,1.9){$\widetilde{C_i \times C_l}$}
\end{pspicture}
\caption{Blow up along $C_i\times C_l$. The strict transforms of $C_i\times C_l$ and $C_j\times C_k$ contains $E$.}\label{fig2}

\end{minipage}

\end{figure}

We denote by $\mathcal{N}(C)$ the collection of reducible nodes of $C$ and for each $i,k \in \{ 1,...,p\}$ let $\mathcal{N}_{i,k}(C)$ denote the subset of $\mathcal{N}(C)^2$ containing every pair of nodes $(R,S)$ such that $R \in C_i$ and $S \in C_k$. Let $\phi:\widetilde{\mathcal{C}^2} \rightarrow \mathcal{C}^2$ be a good partial desingularization of $\mathcal{C}^2$. Denote by $\mathcal{N}_{i,k}(\phi)$ the subset of $\mathcal{N}_{i,k}(C)$ formed by pairs $(R,S)$ such that $\phi^{-1}(R,S) \subset \widetilde{C_i \times C_k}$. 

\begin{proposition} \label{pro:fiber}
Let $\phi:\widetilde{\mathcal{C}^2} \rightarrow \mathcal{C}^2$ be a good partial desingularization. Let $\rho:\widetilde{\mathcal{C}^2} \rightarrow \mathcal{C}$ denote its composition with the first projection $p_1:\mathcal{C}^2 \rightarrow \mathcal{C}$.  Let $R \in C$ and $\widehat{C} := \rho^{-1}(R)$. Let $\mu:\widehat{C} \rightarrow C$ be the restriction to $\widehat{C}$ of $\phi$ composed with the second projection $p_2:\mathcal{C}^2 \rightarrow \mathcal{C}$. Then, the following statements hold:
\begin{enumerate}[1)]
\item $\rho$ is flat.
\item $\widetilde{\mathcal{C}}^2$ is regular along each smooth rational curve of $\widehat{C}$ contracted by $\mu$.
\item If $R$ is not a node of $C$, then $\mu$ is an isomorphism.
\item If $R$ is an irreducible node of $C$, then $\widehat{C}$ is $C$-isomorphic to $C_R$.
\item If $R$ is a reducible node of $C$, then $\widehat{C}$ is $C$-isomorphic to $C(1)$.
\end{enumerate}
Furthermore, for each $i,k \in \{1,...,p\}$, let $D_{i,k}$ denote the strict transform to $\widetilde{\mathcal{C}}^2$ of $C_i \times C_k$. Consider the Cartier divisor
\[D=\sum_{i,k} w_{i,k}D_{i,k}\]
for given integers $w_{i,k}$. Then, if $R$ is a reducible node of $\widehat{C}$, the resctriction $\left. \Ocal_{\widetilde{\mathcal{C}}^2}(D) \right\vert_{\widehat{C}}$ is a twister of $\widehat{C}$. More specifically, for each $i=1,...,p$, let $\widehat{C}_i$ be the strict transform to $\widehat{C}$ of $C_i$ via $\mu$, and for each reducible node $S$ of $C$, let $E_S := \mu^{-1}(S)$. Then
\begin{equation}
\left.  \mathcal{O}_{\widetilde{\mathcal{C}^{2}}}\left(  D\right)  \right\vert
_{\widehat{C}}=\mathcal{O}_{\widehat{C}}\left(  \sum_{k=1}^{p}a_{k}%
\widehat{C_{k}}+\sum_{S \in \mathcal{N}(C)}b_{S}E_{S}\right), \label{twister}%
\end{equation}
where $a_k := \sum_{i,k} w_{i,k}$, the sum over the two $i$ such that $R \in C_i$, and $b_S := \sum_{i,k} w_{i,k}$, the sum over two pairs $(i,k)$ such that $(R,S) \in \mathcal{N}_{i,k}(\phi)$.
\end{proposition}
\begin{proof}
\cite[Proposition 4.4, p.2939]{CoEP}
\end{proof} \\

For instance, let $C = C_1 \cup C_2$ with $C_1 \cap C_2=\{S\}$ and let $f:\mathcal{C} \rightarrow B$ be a regular smoothing of $C$. Suppose that $\phi: \widetilde{\mathcal{C}}^2 \rightarrow \mathcal{C}^2$ is a blowup of $\mathcal{C}^2$ and consider a Cartier divisor
\[
D=w_{1,1}D_{1,1}+w_{1,2}D_{1,2}+w_{2,1}D_{2,1}+w_{2,2}D_{2,2}.
\]
If $\phi$ is the blowup along $C_1 \times C_1$, then
\[
\left. \Ocal_{\widetilde{\mathcal{C}}^2}(D) \right\vert_{\widehat{C}} \simeq \Ocal_{\widehat{C}} \left( (w_{1,1}+w_{2,1})\widehat{C}_1+(w_{1,2}+w_{2,2})\widehat{C}_2+(w_{1,1}+w_{2,2})E_S \right).
\]
If $\phi$ is the blowup along $C_1 \times C_2$, then
\[
\left. \Ocal_{\widetilde{\mathcal{C}}^2}(D) \right\vert_{\widehat{C}} \simeq \Ocal_{\widehat{C}} \left( (w_{1,1}+w_{2,1})\widehat{C}_1+(w_{1,2}+w_{2,2})\widehat{C}_2+(w_{1,2}+w_{2,1})E_S \right).
\]

By definition, $\widehat{C}$ contains a copy
of each component $C_{i}$ of $C$ and contains an exceptional component glued
at each node of $C$ (see figure \ref{fig3})%

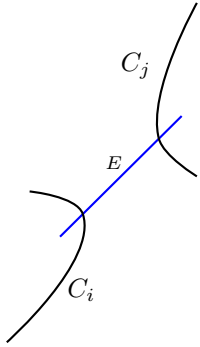
\begin{figure}[!h]

\centering
\begin{pspicture}(1,0)(3.5,4.8)
\psline[linecolor=blue](1.7,1.7)(3.3,3.3)
\pscurve(1.3,2.3)(2,2)(1,0.3)
\put(1.8,0.9){$C_i$}
\pscurve(3.5,2.5)(3,3)(3.5,4.8)
\put(2.5,3.9){$C_j$}
\put(2.3,2.6){{\scriptsize $E$}}
\end{pspicture}
\caption{$\widehat{C}$ locally around an $\phi$-exceptional component $E$.}\label{fig3}
\end{figure}


\subsection{Admissible sheaves}


Recall the notation of preceding Sections. Let
$P\in C$ be a fixed smooth point. For each $i\in\left\{  1,...,p\right\}  ,$
let $Q_{i}\in C_{i}$ be a smooth point. For a line bundle $L$ on $C$ consider
the sheaf
\[
L_{i}:=m_{Q_{i}}\otimes L\text{.}%
\]
According to \cite[p. 12-13]{CoEP} and \cite[Lemmas 30 and 31, p. 3068]{E01}, there exists a unique twister
$T_{i}:=\mathcal{O}_{C}(Z_{i})$ where $Z_{i}$ is a formal sum of
components of $C$ such that $L_{i}\otimes T_{i}$ is $P$-quasistable.

For $i\in\left\{  1,...,p\right\}  $, set
$\widetilde{T}_{i}:=\mathcal{O}_{\widetilde{\mathcal{C}^{2}}}\left(
\widetilde{C_{i}\times Z_{i}}\right)  $ and $\widetilde{\mathcal{T}}=\bigotimes
\limits_{i=1}^{n}\widetilde{T_{i}}.$ If $\phi$ is a good partial desingularization, then $\widetilde{C_{i}\times Z_{i}}$ is a Cartier divisor of $\widetilde
{\mathcal{C}^{2}}$ for all $i\in\left\{  1,...,p\right\}  $. So, $\widetilde{\mathcal{T}}$ is an invertible sheaf of $\widetilde{\mathcal{C}^{2}}$.

Recall the definition of the sheaf on $\widetilde{\mathcal{C}^{2}}/\mathcal{C}$
\begin{equation}
\widetilde{\mathcal{M}}:=\mathcal{I}_{\widetilde{\Delta}/\widetilde{\mathcal{C}^{2}}%
}\otimes\left(  p_{2}\phi\right)  ^{\ast}\mathcal{L}\otimes\widetilde
{\mathcal{T}}, \nonumber%
\end{equation}
where $\mathcal{I}_{\widetilde{\Delta}/\widetilde{\mathcal{C}^{2}}}$ and $\widetilde
{\mathcal{T}}$ are invertible. We will show that $\phi_*\widetilde{\mathcal{M}}$ is a relatively torsion-free, rank-$1$ sheaf. For this we need the notion of admissibility introduced by \cite{EP} and \cite{CoEP}.

Let $\gamma:\mathcal{X} \rightarrow \mathcal{S}$ be a family of connected curves and $\psi: \mathcal{Y} \rightarrow \mathcal{X}$ be a proper morphism such that the composition $\theta := \gamma \circ \psi$ is another family of curves. We say that $\psi$ is a \textit{semistable modification} of $\gamma$ if for each geometric point $s \in \mathcal{S}$, there are a collection of nodes $\mathcal{N}_s$ of the fiber $\mathcal{X}_s$ 
 and a map $\eta_s :\mathcal{N}_s \rightarrow \mathbb{N}$ such that the induced map $\psi_s:\mathcal{Y}_s \rightarrow \mathcal{X}_s$, where $\mathcal{Y}_s$ is the fiber of $\theta$ over $s$, is $\mathcal{X}_s$-isomoprhic to $\mu_{\eta _s}:(\mathcal{X}_s)_{\eta _s} \rightarrow \mathcal{X}_s$. If $\eta_s$ is constant and equal to $1$ for every $s$, we say that $\psi$ is a \textit{small semistable modification} of $\gamma$.

Assume $\psi$ is a semistable modification of $\gamma$. Let $\mathcal{L}$ be an invertible sheaf on $\mathcal{Y}$. We say that $\mathcal{L}$ is $\psi$-\textit{admissible} (resp. \textit{negatively} $\psi$-\textit{admissible}, resp. \textit{positively} $\psi$-\textit{admissible}, resp. $\psi$-\textit{invertible}) at a given geometric point $s \in \mathcal{S}$ if the restriction of $\mathcal{L}$ to every chain of rational curves of $\mathcal{Y}_s$ over a node of $\mathcal{X}_s$ has degree $-1$, $0$ or $1$. (resp. $-1$ or $0$, resp. $0$ or $1$, resp. $0$). We say that $\mathcal{L}$ is $\psi$-\textit{admissible} (resp. \textit{negatively} $\psi$-\textit{admissible}, resp. \textit{positively} $\psi$-\textit{admissible}, resp. $\psi$-\textit{invertible}) if $\mathcal{L}$ is so at every $s \in \mathcal{S}$. Notice that, if $\mathcal{L}$ is negatively (resp. positively) $\psi$-admissible, for every chain of rational curves of $\mathcal{Y}_s$ over a node of $\mathcal{X}_s$, the degree of $\mathcal{L}$ on each component of the chain is $0$ but for at most one component where the degree is $-1$ (resp. $1$).

\begin{proposition} \label{pro:admissible} 
Let $\gamma:\mathcal{X} \rightarrow \mathcal{S}$ be a family of connected curves, $\psi: \mathcal{Y} \rightarrow \mathcal{X}$ a semistable modification of $\gamma$ and $\theta := \gamma \circ \psi$. Let $\mathcal{L}$ be an invertible sheaf on $\mathcal{Y}$ of relative degree $d$ over $\mathcal{S}$. Then the following statements hold:
\begin{enumerate}[1)]
\item The points $s \in \mathcal{S}$ at which $\mathcal{L}$ is $\psi$-\textit{admissible} (resp. negatively $\psi$-admissible, resp. positively $\psi$-admissible, resp. $\psi$-invertible) form an open set of $\mathcal{S}$.
\item $\mathcal{L}$ is $\psi$-\textit{admissible} if and only if $\psi_*\mathcal{L}$ is a relatively torsion-free, rank-$1$ sheaf on $\mathcal{X}/\mathcal{S}$ of relative degree $d$, whose formation commutes with base change. In this case, $\text{R}^1\psi_*\mathcal{L}=0$.
\item If $\mathcal{L}$ is $\psi$-admissible then the evaluation map $v:\psi^*\psi_*\mathcal{L} \rightarrow \mathcal{L}$ is surjective if and only if $\mathcal{L}$ is positively $\psi$-admissible. Furthermore, $v$ is bijective if and only if $\mathcal{L}$ is $\psi$-invertible, if and only if $\psi_*\mathcal{L}$ is invertible.
\end{enumerate}
\end{proposition} 
\begin{proof}
\cite[Theorem 3.1, p.63]{EP}
\end{proof} \\

\begin{proposition} \label{pro:stability}
Let $X$ be a curve and $\psi:Y \rightarrow X$ a semistable modification of $X$. Let $P$ be a single point of $Y$ not lying on any component contracted by $\psi$. Let $\mathcal{E}$ be a locally free sheaf on $X$ and $\mathcal{L}$ an invertible sheaf on $Y$. Then $\mathcal{L}$ is semistable (resp. $P$-quasistable, resp. stable) with respect to $\psi ^*\mathcal{E}$ if and only if $\mathcal{L}$ is $\psi$-admissible (resp. negatively $\psi$-admissible, resp. $\psi$-invertible) and $\psi _* \mathcal{L}$ is semistable (resp. $\psi(P)$-quasistable, resp. stable) with respect to $\mathcal{E}$.
\end{proposition}
\begin{proof}
\cite[Theorem 4.1, p.70]{EP}
\end{proof} \\

According to Propositions \ref{pro:admissible} and \ref{pro:stability}, to conclude that $\phi_*\widetilde{\mathcal{M}}$ is a rank-$1$, torsion-free sheaf, it suffices to show that $\deg_E \widetilde{\mathcal{M}} \in \{-1,0,1\}$ for all $\phi$-exceptional component $E \subset \widehat{C}$.

Notice that $\left. (p_{2}\phi)^*\mathcal{L}\right\vert_{\widehat{C}}$ has degree $0$ on a $\phi$-exceptional component $E$. In the next lemma we calculate the degree of $\mathcal{I}_{\widetilde{\Delta}/\widetilde{\mathcal{C}^2}}$ on each exceptional component $E \subset \widehat{C}$.

\begin{lemma} \label{lem:diag}
Let $\phi:\widetilde{\mathcal{C}^2} \rightarrow \mathcal{C}^2$ be a good partial desingularization of $\mathcal{C}^2$. If $E$ is a $\phi$-exceptional component such that $\phi(E)=(R,R)$, where $R$ is a node of $C$, then $\deg_E (\mathcal{I}_{\widetilde{\Delta}/\widetilde{\mathcal{C}^2}})=1$.
\end{lemma}
\begin{proof}
By definition of good partial desingularization, the exceptional component $E$ is obtained after blowing-up the diagonal subscheme $\Delta$. Suppose $R \in C_i \cap C_j$. By Proposition \ref{pro:strict}, $\widetilde{\Delta}$ contains $E$ and, equivalently, $E \subset \widetilde{C_i \times C_j}$. The figure \ref{fig4} illustrates this situation:

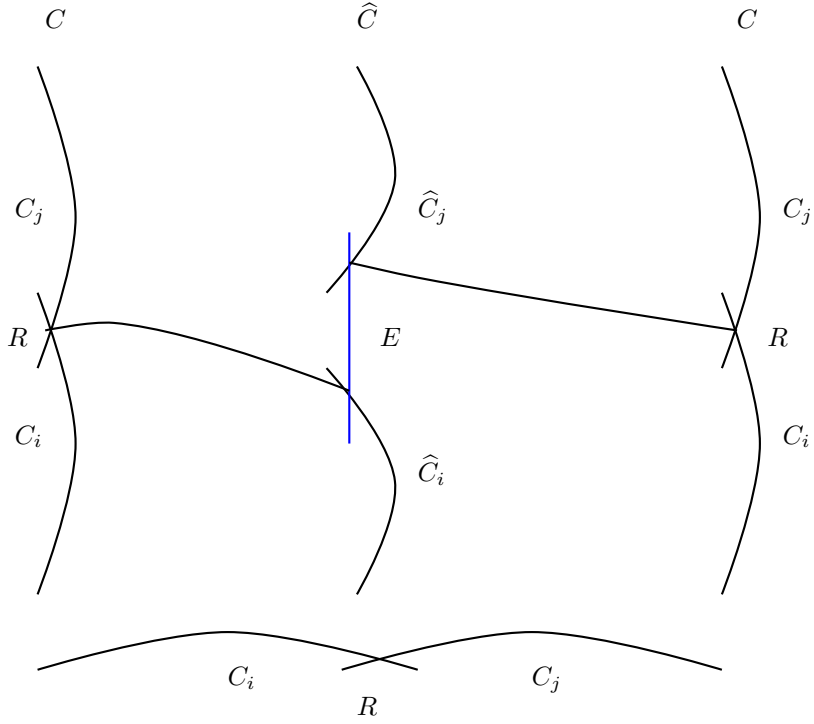
\begin{figure}[!h]
\centering
\begin{pspicture}(0,0)(10,10)
\pscurve(0.5,1)(3,1.5)(5.5,1)
\put(3,0.8){$C_i$}
\pscurve(4.5,1)(7,1.5)(9.5,1)
\put(7,0.8){$C_j$}
\put(4.7,0.4){$R$}
\pscurve(0.5,2)(1,4)(0.5,6)
\pscurve(0.5,5)(1,7)(0.5,9)
\put(0.2,4){$C_i$}
\put(0.1,5.3){$R$}
\put(0.2,7){$C_j$}
\pscurve(9.5,2)(10,4)(9.5,6)
\pscurve(9.5,5)(10,7)(9.5,9)
\put(10.3,4){$C_i$}
\put(10.1,5.3){$R$}
\put(10.3,7){$C_j$}
\pscurve(4.7,2)(5.2,3.5)(4.3,5)
\pscurve(4.3,6)(5.2,7.5)(4.7,9)
\put(5.5,3.5){$\widehat{C}_i$}
\put(5,5.3){$E$}
\put(5.5,7){$\widehat{C}_j$}
\psline[linecolor=blue](4.6,4)(4.6,6.8)
\pscurve(0.6,5.5)(1.5,5.6)(4.6,4.7)
\pscurve(4.6,6.4)(5.5,6.2)(9.7,5.5)
\put(0.6,9.5){$C$}
\put(4.7,9.5){$\widehat{C}$}
\put(9.7,9.5){$C$}
\end{pspicture}
\caption{Local blowup.}\label{fig4}
\end{figure}

We can see that $C_i$ degenerates to $\widehat{C}_i$ on the left side and degenerates to $\widehat{C}_i \cup E$ on the right side. Using the degeneration of $C_i$ to $\widehat{C}_i$ and $\deg_{C_i}(\mathcal{I}_{\widetilde{\Delta}/\widetilde{\mathcal{C}^2}})=-1$, it follows that 
\[
-1=\deg_{C_i}(\mathcal{I}_{\widetilde{\Delta}/\widetilde{\mathcal{C}^2}})=\deg_{\widehat{C}_i}(\mathcal{I}_{\widetilde{\Delta}/\widetilde{\mathcal{C}^2}}).
\]

On the other hand, using the degeneration of $C_i$ to $\widehat{C}_i \cup E$ we get
\[
0=\deg_{C_i}(\mathcal{I}_{\widetilde{\Delta}/\widetilde{\mathcal{C}^2}})=\deg_{\widehat{C}_i \cup E}(\mathcal{I}_{\widetilde{\Delta}/\widetilde{\mathcal{C}^2}},)
\]
so, $\deg_{E}(\mathcal{I}_{\widetilde{\Delta}/\widetilde{\mathcal{C}^2}})=1$. The proof is complete.
\end{proof} \\

To complete the analysis of the degree of the invertible sheaf $\widetilde{\mathcal{M}}$ over the $\phi$-exceptional components contained in $\widehat{C}$, it remains to analyze the sheaf $\widetilde{\mathcal{T}}$.

\begin{lemma}
\label{lem:degzero}Let $E$ be a $\phi$-exceptional component contained in the
fiber $\widehat{C}=\rho^{-1}\left(  R\right)  ,$ where $R\in C_{i}\cap C_{j}$
is a node of $C$. Then,

\begin{description}
\item[(a)] $\deg_{E}\left[  \mathcal{O}_{\widetilde{\mathcal{C}^{2}}}\left(
\widetilde{C_{i}\times Z_{i}}+\widetilde{C_{j}\times Z_{i}}\right)  \right]
=0;$

\item[(b)] $\deg_{E}\left[  \mathcal{O}_{\widetilde{\mathcal{C}^{2}}}\left(
\widetilde{C_{i}\times Z_{i}}+\widetilde{C_{j}\times Z_{j}}\right)  \right]
=\deg_{E}\left[  \mathcal{O}_{\widetilde{\mathcal{C}^{2}}}\left(
\widetilde{C_{j}\times Z_{j}}-\widetilde{C_{j}\times Z_{i}}\right)  \right]  $.
\end{description}
\end{lemma}

\begin{proof}
(a) Assume that $E$ is such that $\phi(E)=(R,S)$, with $S\in C_{k}\cap$ $C_{l}$. Suppose that $Z_{i}=a_{l}C_{l}+a_{k}C_{k}+...$, where $a_{l},a_{k}\in\mathbb{Z}$.
Without loss generality, suppose that $E$ is contained in the strict transform
of $C_{i}\times C_{l}$ and that $E$ is not contained in the strict transform
of $C_{s}\times C_{l}$. According to Proposition \ref{pro:fiber}, locally around $E$, we have
\[
\left.  \mathcal{O}_{\widetilde{\mathcal{C}^{2}}}\left(  \widetilde
{C_{i}\times Z_{i}}\right)  \right\vert _{\widehat{C}}=\mathcal{O}%
_{\widehat{C}}(\cdots+a_{l}\widehat{C_{l}}+a_{l}E+a_{k}\widehat{C_{k}}+\cdots)
\]
and
\[
\left.  \mathcal{O}_{\widetilde{\mathcal{C}^{2}}}\left(  \widetilde
{C_{j}\times Z_{i}}\right)  \right\vert _{\widehat{C}}=\mathcal{O}%
_{\widehat{C}}(\cdots+a_{l}\widehat{C_{l}}+a_{k}E+a_{k}\widehat{C_{k}}%
+\cdots).
\]
So
\[
\deg_{E}\left[  \mathcal{O}_{\widetilde{\mathcal{C}^{2}}}\left(
\widetilde{C_{i}\times Z_{i}}\right)  \right]  =a_{k}-a_{l}%
\]
and%
\[
\deg_{E}\left[  \mathcal{O}_{\widetilde{\mathcal{C}^{2}}}\left(
\widetilde{C_{j}\times Z_{i}}\right)  \right]  =a_{l}-a_{k}.
\]

Hence%
\[
\deg_{E}\left[  \mathcal{O}_{\widetilde{\mathcal{C}^{2}}}\left(
\widetilde{C_{i}\times Z_{i}}+\widetilde{C_{j}\times Z_{i}}\right)  \right]
=0.
\]

(b) Since%
\[
\widetilde{C_{i}\times Z_{i}}+\widetilde{C_{j}\times Z_{j}}=\left(
\widetilde{C_{i}\times Z_{i}}+\widetilde{C_{j}\times Z_{i}}\right)  +\left(
\widetilde{C_{j}\times Z_{j}}-\widetilde{C_{j}\times Z_{i}}\right)  \text{,}%
\]
the result follows from (a).
\end{proof}

\bigskip

The line bundle $\mathcal{O}_{\widetilde{\mathcal{C}^{2}}}\left(
\widetilde{C_{j}\times Z_{j}}-\widetilde{C_{j}\times Z_{i}}\right)  $ over
$\widetilde{\mathcal{C}^{2}}$ is called the \textit{twister difference} with respect to $i$ and $j$ and it will denoted by $\widetilde{T}_{j-i}$


\section{Twister difference}


Let $C$ be a connected, nodal curve with components $C_{1},...,C_{p}$. Let
$P\in C$ be a fixed smooth point. 


Recall from Section 2 if $I$ is an invertible sheaf on $C$, $E$ is a
polarization on $C$ and $Y$ is a subcurve of $C$, then
\[
\beta_{Y}(I)=\chi(I_{Y})+\frac{\deg_{Y}(E)}{\text{rk}(E)}\text{.}%
\]

\bigskip

\begin{lemma}
\label{lem:beta}Let $I$ be an invertible sheaf on $C$ and let $E$ be a
polarization on $C$. If $Y$ and $Z$ are subcurves of $C$ such that $\dim Y\cap
Z=0$ or $Y \cap Z=\emptyset$, then%
\[
\beta_{Y\cup Z}(I)=\beta_{Y}(I)+\beta_{Z}(I)-\delta_{YZ},
\]
where $\delta_{YZ}=\#(Y\cap Z)$.
\end{lemma}

\begin{proof}
The proof when $Y \cap Z=\emptyset$ is trivial. If $\dim Y\cap
Z=0$, it follows from the exact sequence%
\[
0\rightarrow I_{Z}(-Y\cap Z)\rightarrow I_{Y\cup Z}\rightarrow I_{Y}%
\rightarrow0
\]
that $\chi(I_{Y\cup Z})=\chi(I_{Y})+\chi(I_{Z})-\delta_{YZ}$. As $E$ is a locally free sheaf, it follows that
\[
\frac{\deg_{Y\cup Z}(E)}{\text{rk}(E)}=\frac{\deg_{Y}(E)}{\text{rk}(E)}%
+\frac{\deg_{Z}(E)}{\text{rk}(E)}%
\]
and so
\begin{align*}
\beta_{Y\cup Z}(I)  & =\chi(I_{Y})+\chi(I_{Z})-\delta_{YZ}+\frac{\deg_{Y}%
(E)}{\text{rk}(E)}+\frac{\deg_{Z}(E)}{\text{rk}(E)}\\
& =\beta_{Y}(I)+\beta_{Z}(I)-\delta_{YZ}%
\end{align*}
and the proof is complete.
\end{proof}

\bigskip

Let $C_{i}$ and $C_{j}$ be irreducible components
of $C$ such that $\delta_{i,j}>0$. Consider $Q_{i}\in C_{i},$ $Q_{j}\in C_{j}$
and $P\in C_{1}$ nonsingular points of $C$. Let $T_{i}=\left.\Ocal_{\mathcal{C}}(Z_i)\right\vert _C$ and $T_{j}=\left.\Ocal_{\mathcal{C}}(Z_j)\right\vert _C$
be the twisters such that
\[
M_{i}:=m_{Q_{i}}\otimes L\otimes T_{i}\text{ \ and \ }M_{j}:=m_{Q_{j}}\otimes
L\otimes T_{j}%
\]
are $P$-quasistable with respect to $E$.

Consider%
\[
M:=m_{Q_{j}}\otimes L\otimes T_{i}.
\]
If $Y$ is a subcurve of $C$, we have three possibilities:
\begin{eqnarray}
\beta_{Y}(M)=\beta_{Y}(M_{i})-1,  \text{ if } Y\supset C_{j} \text{ and } Y\nsupseteq
C_{i} \label{beta1}\\
\beta_{Y}(M)=\beta_{Y}(M_{i})+1,  \text{ if } Y\supset C_{i} \text{ and } Y\nsupseteq
C_{j} \label{beta2}\\
\beta_{Y}(M)=\beta_{Y}(M_{i}),  \text{ otherwise}. \label{beta3}
\end{eqnarray}
Let $A$ and $B$ be the following sets of subcurves of $C$:%
\begin{align*}
A  & =\{Y;\beta_{Y}(M)<0\text{ and }\beta_{Y}(M)\leq\beta_{X}(M),\text{
}\forall X\text{ subcurve of }C\};\\
B  & =\{Y;\beta_{Y}(M)=0\text{, }Y\supseteq C_{1}\text{ and }\beta_{Y}%
(M)\leq\beta_{X}(M),\text{ }\forall X\text{ subcurve of }C\}.
\end{align*}
Note that if $A$ and $B$ are empty, then $M$ is $P$-quasistable. We now define
a subcurve $Z_{i,j}$ of $C$ in the following way. If $A\neq\emptyset$, then let $Z_{i,j}$ be a subcurve of $A$ with minimal number of components. If $A=\emptyset$ and $B\neq\emptyset$, we let $Z_{i,j}$ be a subcurve of $B$ with
minimal number of component. 

Notice that, in all situations, $Z_{i,j}$ contains $C_{j}$ and does not contain $C_{i}$.

\bigskip

\begin{proposition}
\label{pro:subcurve}With the notation fixed above, the sheaf $M\otimes
\mathcal{O}_{C}(-Z_{i,j})$ is \ $P$-quasistable.
\end{proposition}

\begin{proof}
We consider only the case $A\neq\emptyset.$ The case $A=\emptyset$ and
$B\neq\emptyset$ is similar. Define $M^{\prime}=M\otimes\mathcal{O}_{C}(-Z_{i,j})$.

Let $Y$ be a subcurve of $C$. We can write
\[Y=Y_1 \cup Z_1,\]
with $Z_{1}=\overline{Z_{i,j} \cap Y}$ and $Y_{1}=\overline{Y-Z_{i,j}}$. Note that, by the
minimality of $\beta_{Z_{i,j}}(M)$ and by Lemma \ref{lem:beta}, we have
\[
\beta_{Z_{i,j}}(M)\leq\beta_{Z_{i,j}\cup Y_{1}}(M)=\beta_{Z_{i,j}}(M)+\beta_{Y_{1}}%
(M)-\delta_{Y_{1},Z_{i,j}}.
\] 
So, 
\begin{equation}
\beta_{Y_1}(M)-\delta_{Y_1,Z_{i,j}}\geq 0. \label{in1}
\end{equation}

Now,
\begin{align}
\beta_{Y}(M^{\prime})  & =\beta_{Y}(M)+\delta_{Z_{1}Z^{\prime}_{i,j}}-\delta_{Y_{1},Z_{i,j}}\\
& =\beta_{Z_{1}}(M)+\beta_{Y_{1}}(M)-\delta_{Y_{1}Z_{1}}+\delta_{Z_{1}Z^{\prime}_{i,j}}-\delta_{Y_{1},Z_{i,j}}\\
& =(\beta_{Y_{1}}(M)-\delta_{Y_{1}Z_{i,j}})+(\beta_{Z_{1}}(M)+\delta_{Z_{1}Z^{\prime}_{i,j}}-\delta_{Y_{1}Z_{1}})\\
& \geq (\beta_{Y_{1}}(M)-\delta_{Y_{1}Z_{i,j}})+(\beta_{Z_{i,j}}(M)+\delta_{Z_{1}Z^{\prime}_{i,j}}-\delta_{Y_{1}Z_{1}}) \label{in4},
\end{align}
where inequality (\ref{in4}) follows from the minimal property of $\beta_{Z_{i,j}}(M)$. 
Note that, by (\ref{in1}), the first parenthesis in (\ref{in4}) is nonnegative.

We have two cases to consider.

\bigskip

\textbf{(I)} $C_i$ is not a component of $Y$.

If $Z_1=\emptyset$ then, $Y$ does not contain $C_j$. So, $\delta_{Z_1Z^{\prime}_{i,j}}=\delta_{Y_1Z_{i,j}}=0$ and we conclude that $\beta_Y(M^{\prime})=\beta_Y(M) \geq 0$.
 
Suppose that $Z_1 \neq \emptyset$. In this case, since $C_j \subset Z_{i,j}$, we have $C_i \cap Z_{i,j} \neq \emptyset$ and so,
\[\delta_{Z_1Z^{\prime}_{i,j}}-\delta_{Y_{1}Z_{1}} \geq 1.\]
By (\ref{beta1}), we have $\beta_{Z_{i,j}}(M) \geq -1$ and we conclude that the second parenthesis in (\ref{in4}) is non negative too. Thus $\beta_{Y}(M^{\prime}) > 0 $.

Now, if $Y \supset C_1$ and $C_1 \subset Z_{i,j}$ then, by  (\ref{beta1}), $\beta_{Z_{i,j}}(M) > -1$, so, the second parenthesis in (\ref{in4}) is positive and we conclude that $\beta_{Y}(M^{\prime}) > 0 $. If $Y \supset C_1$, $C_1 \not\subset Z_{i,j}$ and $\beta_{Z_{i,j}}(M) > -1$ then, we are done, as above. If $Y \supset C_1$, $C_1 \not\subset Z_{i,j}$ and $\beta_{Z_{i,j}}(M) = -1$ then, since
\[C_1 \subseteq Z_{i,j} \cup Y_1,\]
it follows from (\ref{beta1}) we have $\beta_{Z_{i,j}\cup Y_1}(M) > -1$ and consequently (\ref{in1}) is strict. So, the first parenthesis in (\ref{in4}) is also strict. Thus, $\beta_{Y}(M^{\prime}) \geq 0 $ in any case.

\bigskip

\textbf{(II)} $C_i$ is a component of $Y$.

In this case, the subcurve $Y_1$ is nonempty and the subcurve $Z_{i,j} \cup Y_1$ contains $C_i$ and $C_j$. Thus, by (\ref{beta3}) we have 
\[\beta_{Z_{i,j} \cup Y_1}(M)=\beta_{Z_{i,j}}(M)+\beta_{Y_{1}}%
(M)-\delta_{Y_{1},Z_{i,j}} \geq 0\]
and, by (\ref{beta3}), the inequality is strict if $Y_1 \supset C_1$.

Writing the right side of (\ref{in4}) in the form
\[
(\beta_{Z_{i,j}}(M)+\beta_{Y_{1}}(M)-\delta_{Y_{1}Z_{i,j}})+(\delta_{Z_{1}Z^{\prime}_{i,j}}-\delta_{Y_{1}Z_{1}})
\]
we obtain, by (\ref{in4}), that
\begin{equation}
\beta_{Y}(M^{\prime}) \geq (\beta_{Z_{i,j}}(M)+\beta_{Y_{1}}(M)-\delta_{Y_{1}Z_{i,j}})+(\delta_{Z_{1}Z^{\prime}_{i,j}}-\delta_{Y_{1}Z_{1}}). \label{in5}
\end{equation}

Since the two parenthesis in (\ref{in5}) are non negative, we have $\beta_{Y}(M^{\prime}) \geq 0$. Moreover, if $C_1 \subset Y$ then, the first parenthesis in (\ref{in5}) is strict, so $\beta_{Y}(M^{\prime}) > 0$. Therefore the sheaf $M'$ is $P$-quasistable and the proof is complete.
\end{proof}

\bigskip

\begin{corollary}
Keep the notation of Proposition \ref{pro:subcurve}. Then,
\[
T_j=T_i \otimes \Ocal_{C}({-Z_{i,j})}.
\]
\end{corollary}

\begin{proof}
It follows from Proposition \ref{pro:subcurve} and unicity of the twister.
\end{proof}

\bigskip

\begin{definition} \label{def:subcurve}
The subcurve $Z_{i,j}$ in the Proposition \ref{pro:subcurve} is called the twister difference subcurve between $j$ and $i$. We define $Z_{i,i}=\emptyset$.
\end{definition}

\bigskip

Note that if $M$ is $P$-quasistable then $Z_{i,j}=C$. Let us go back to the analysis of $\widetilde{\mathcal{T}}$. We have the following

\bigskip

\begin{lemma}
\label{lem:deg_ttil}Let $E$ be a $\phi$-exceptional component of $\widehat{C}%
$. Then,
\[
\deg_{E}\widetilde{\mathcal{T}}=\deg_{E}\left( \widetilde{T}_{j-i}\right)  =\deg
_{E}\left[  \mathcal{O}_{\widetilde{\mathcal{C}^{2}}}\left(  \widetilde
{C_{i}\times Z_{i}}+\widetilde{C_{j}\times Z_{j}}\right)  \right]  .
\]

\end{lemma}

\begin{proof}
As we have
\[
\deg_{E}\widetilde{\mathcal{T}}=\deg_{E}\left[  \mathcal{O}_{\widetilde{\mathcal{C}^{2}}}\left(\widetilde{C_{i}\times Z_{i}}+\widetilde{C_{j}\times Z_{j}}\right)  \right] + \deg_{E}\left[ \Ocal_{\widetilde{\mathcal{C}^{2}}}\left(\sum_{r \neq i,j}\widetilde{C_{r}\times Z_{r}}\right)\right],
\]
it is sufficient to prove that $\deg_{E}\left(  \widetilde{C_{r}\times Z_{r}%
}\right)  =0$ if $r\neq i,j$. Suppose $E=\mu^{-1}(S)$ with $S\in C_{k}\cap
C_{l}$ and $k,l \neq i,j$. According to Proposition \ref{pro:fiber}, we have
\[
\left.  \mathcal{O}_{\widetilde{\mathcal{C}^{2}}}\left(  \widetilde
{C_{r}\times Z_{r}}\right)  \right\vert _{\widehat{C}}=\mathcal{O}%
_{\widehat{C}}(\cdots+0\cdot\widehat{C_{k}}+0\cdot E+0\cdot\widehat{C_{l}%
}+\cdots).
\]
So, $\deg_{E}\left(  \widetilde{C_{r}\times Z_{r}}\right)  =0$, if $r\neq i,j$.
\end{proof}

\bigskip

Recall the sheaf $\widetilde{\mathcal{M}}$ defined in (\ref{eqn1}). Note that if $\mathcal{E}$ is a polarization over $\mathcal{C} \rightarrow B$ then, the vector bundle $\widetilde{\mathcal{E}}=(p_1\phi)^*\mathcal{E}$ is a polarization on the family 
\begin{displaymath}
\xymatrix{
\widetilde{\mathcal{C}^{2}} \ar[r]^{\phi} & \mathcal{C}^{2} \ar[r]^{p_1} & \mathcal{C}. }
\end{displaymath} 
So, we have $\deg_{E}(\widetilde{\mathcal{E}})=0$ and $\beta_E(\widetilde{\mathcal{M}})=\deg_E(\widetilde{\mathcal{M}})+1$ for every $\phi$-exceptional component $E$ of $\widehat{C}$. 

\bigskip

\begin{proposition}
\label{lem:RS_igual}Let $E$ be a $\phi$-exceptional component of $\widehat{C}%
$ with $R \in C_i \cap C_j$ and $E=\phi^{-1}(R,R)$. The following properties hold
\begin{enumerate}[(i)]
\item if $\Ocal_C(Z_{i,j}) \simeq \Ocal_C$ then, $\deg_E{\widetilde{\mathcal{M}}}=1$ and $\beta_E(\widetilde{\mathcal{M}})=2$;
\item otherwise $\deg_E{\widetilde{\mathcal{M}}}=0$ and $\beta_E(\widetilde{\mathcal{M}})=1$
\end{enumerate}
\end{proposition}

\begin{proof}
(i) Since $\Ocal_C(Z_{i,j}) \simeq \Ocal_C$, we have $\deg_E \widetilde{\mathcal{T}}=\deg_E \widetilde{T}_{j-i}=0$ and hence $\deg_E \widetilde{\mathcal{M}}=\deg_E \mathcal{I}_{\widetilde{\Delta}/\widetilde{\mathcal{C}^2}}$. By Lemma \ref{lem:diag}, we have $\deg_E \mathcal{I}_{\widetilde{\Delta}/\widetilde{\mathcal{C}^2}}=1$, from which we get
\[
\deg_E{\widetilde{\mathcal{M}}}=1 \text{ and } \beta_E(\widetilde{\mathcal{M}})=2
\]

(ii) In this case the subcurve $Z_{i,j}$ contains $C_{j}$ and does not contain $C_{i}$. By Proposition \ref{pro:strict}, we have
\[
E \subset \widetilde{C_i \times C_j} \text{ and } E \subset \widetilde{C_j \times C_i}.
\] 
So, by Proposition \ref{pro:fiber},
\[
\left.  \mathcal{O}_{\widetilde{\mathcal{C}^{2}}}(-\widetilde{C_{j}\times
Z_{i,j}})\right\vert _{\widehat{C}}=\mathcal{O}_{\widehat{C}}(\cdots0\cdot
\widehat{C_{i}}+0\cdot E-1\cdot\widehat{C_{j}}+\cdots).
\]
Hence 
\[
\deg_{E}\widetilde{\mathcal{M}}=\deg_{E}\mathcal{I}_{\widetilde
{\Delta}/\widetilde{\mathcal{C}^{2}}}+\deg_{E}\widetilde{\mathcal{T}}=1+(-1)=0
\]
and $\beta_{E}(\widetilde{\mathcal{M}})=1$.
\end{proof}

\bigskip

\begin{proposition}
\label{lem:RS_dif}
Let $E$ be a $\phi$-exceptional component of $\widehat{C}$. Suppose $R\in C_{i}\cap C_{j}$ and $E=\phi^{-1}(R,S)$ with $S \in C_k \cap C_l$ and $S\neq R$. The following properties hold
\begin{enumerate}[(i)]
\item if $\Ocal_C(Z_{i,j}) \simeq \Ocal_C$ then, $\deg_E{\widetilde{\mathcal{M}}}=0$ and $\beta_E(\widetilde{\mathcal{M}})=1$;
\item if $\Ocal_C(Z_{i,j})$ is not trivial then, the following properties hold: 
\begin{enumerate}[(a)]
\item If $C_{k}$ and $C_{l}$ are components of $Z_{i,j}$, then $\deg_E \widetilde{\mathcal{M}}=0$ and $\beta_{E}(\widetilde{\mathcal{M)}}=1$;
\item If $C_{k}$ and $C_{l}$ are components of $Z_{i,j}'$, then $\deg_E \widetilde{\mathcal{M}}=0$ and $\beta_{E}(\widetilde{\mathcal{M)}}=1$;
\item If $C_{k}\subset Z_{i,j}$ and $C_{l}\subset Z_{i,j}'$, then $\deg_E\widetilde{\mathcal{M}} \in \{-1,1\}$ and  $\beta
_{E} (\widetilde{\mathcal{M}})\in \{0,2\}$
\end{enumerate}
\end{enumerate}
\end{proposition}

\begin{proof}
Note that $\deg_{E}\mathcal{I}_{\widetilde{\Delta}/\widetilde{\mathcal{C}^{2}}}=0$. So, in this case, we have
\[
\deg_E \widetilde{\mathcal{M}}=\deg_E \widetilde{\mathcal{T}}=\deg_E \widetilde{T}_{j-i}.
\]

(i) If $\Ocal_C(Z_{i,j}) \simeq \Ocal_C$, then $\deg_E \widetilde{\mathcal{T}}=0$. So, $\deg_E \widetilde{\mathcal{M}}=0$ and $\beta_E(\widetilde{\mathcal{M}})=1$.

(ii) (a) According to Proposition \ref{pro:fiber}, we have two possibilities for $\left.\mathcal{O}_{\widetilde{\mathcal{C}^{2}}}(-\widetilde{C_{j} \times Z_{i,j}%
})\right\vert _{\widehat{C}}$. 

If $E\subset\widetilde{C_{i}\times C_{l}}$ and $E\subset\widetilde{C_{j}\times C_{k}}$ then,%
\[
\left.  \mathcal{O}_{\widetilde{\mathcal{C}^{2}}}(-\widetilde{C_{j}\times
Z_{i,j}})\right\vert _{\widehat{C}}=\mathcal{O}_{\widehat{C}}(\cdots-1\cdot
\widehat{C_{k}}-1\cdot E-1\cdot\widehat{C_{l}}+\cdots).
\]
So, $\deg_{E}\widetilde{\mathcal{M}}=0$ and $\beta_{E}(\widetilde{\mathcal{M}})=1$.

If $E\subset\widetilde{C_{i}\times C_{k}}$ and $E\subset\widetilde
{C_{j}\times C_{l}}$ then,%
\[
\left.  \mathcal{O}_{\widetilde{\mathcal{C}^{2}}}(-\widetilde{C_{j}\times
Z_{i,j}})\right\vert _{\widehat{C}}=\mathcal{O}_{\widehat{C}}(\cdots-1\cdot
\widehat{C_{k}}-1\cdot E-1\cdot\widehat{C_{l}}+\cdots).
\]
So, $\deg_{E}\widetilde{\mathcal{M}}=0$ and $\beta_{E}(\widetilde{\mathcal{M}})=1$.

(b) In this case we have%
\[
\left.  \mathcal{O}_{\widetilde{\mathcal{C}^{2}}}(-\widetilde{C_{j}\times
Z_{i,j}})\right\vert _{\widehat{C}}=\mathcal{O}_{\widehat{C}}(\cdots+0\cdot
\widehat{C_{k}}+0\cdot E+0\cdot\widehat{C_{l}}+\cdots).
\]
So, $\deg_{E}\widetilde{\mathcal{M}}=0$ and $\beta_{E}(\widetilde{\mathcal{M}})=1.$

(c) In this case, we have two possibilities for $\left.  \mathcal{O}%
_{\widetilde{\mathcal{C}^{2}}}(-\widetilde{Z_{i,j}\times C_{j}})\right\vert
_{\widehat{C}}$. 

If $E\subset\widetilde{C_{i}\times C_{k}}$ and $E\subset\widetilde
{C_{j}\times C_{l}}$ then,
\[
\left.  \mathcal{O}_{\widetilde{\mathcal{C}^{2}}}(-\widetilde{C_{j}\times
Z_{i,j}})\right\vert _{\widehat{C}}=\mathcal{O}_{\widehat{C}}(\cdots-1\cdot
\widehat{C_{k}}+0\cdot E+0\cdot\widehat{C_{l}}+\cdots).
\]
So, $\deg_{E}\widetilde{\mathcal{M}}=-1$ and $\beta_{E}(\widetilde{\mathcal{M}})=0$.

If $E\subset\widetilde{C_{i}\times C_{l}}$ and $E\subset\widetilde
{C_{j}\times C_{k}}$ then,
\[
\left.  \mathcal{O}_{\widetilde{\mathcal{C}^{2}}}(-\widetilde{C_{j}\times
Z_{i,j}})\right\vert _{\widehat{C}}=\mathcal{O}_{\widehat{C}}(\cdots-1\cdot
\widehat{C_{k}}-1\cdot E+0\cdot\widehat{C_{l}}+\cdots).
\]
So, $\deg_{E}\widetilde{\mathcal{M}}=1$ and $\beta_{E}(\widetilde{\mathcal{M}})=2$.
 
Thus, the proof is complete.
\end{proof}

\bigskip

\begin{proposition}
\label{cor:betas}Keep the notation of Proposition \ref{lem:RS_dif} and suppose that $\Ocal_C(Z_{i,j})$ is not trivial. Let $Y$ be a subcurve of $C$ and let $E$ be a
$\phi$-exceptional component of $\widehat{C}$. Suppose $E=\phi^{-1}(R,S)$, with $S \in Y \cap Y'$ and $R \neq S$.
\begin{enumerate}[(a)]
\item If $Y\subset Z_{i,j}$ and $E\subset\widetilde{C_{i}\times Y}$, then
$\beta_{E}(\widetilde{\mathcal{M}})\in\{0,1\}$;
\item If $Y\subset Z_{i,j}$ and $E\subset\widetilde{C_{j}\times Y}$, then
$\beta_{E}(\widetilde{\mathcal{M}})\in\{1,2\}$;
\item If $Y\subset Z_{i,j}'$ and $E\subset\widetilde{C_{i}\times Y}$, then
$\beta_{E}(\widetilde{\mathcal{M}})\in\{1,2\}$;
\item If $Y\subset Z_{i,j}'$ and $E\subset\widetilde{C_{j}\times Y}$, then
$\beta_{E}(\widetilde{\mathcal{M}})\in\{0,1\}$.
\end{enumerate}
\end{proposition}

\begin{proof}
Let $C_k$ and $C_l$ be irreducibles components of $C$ such that $C_k \subset Y$, $C_l \subset Y'$ and $S \in C_k \cap C_l$.

(a) If $C_{k}$ and $C_{l}$ are components of $Z_{i,j}$, then, by Proposition
\ref{lem:RS_dif}(ii)(a), $\beta_{E}(\widetilde{\mathcal{M}})=1$. If $C_{k}$ and $C_{l}$ are components of $Z_{i,j}$ and $Z_{i,j}'$, respectively, then, by hypothesis, $E\subset\widetilde{C_{i}\times C_{k}}$ and according to Proposition \ref{lem:RS_dif}(ii)(c), we have $\beta_{E}(\widetilde{\mathcal{M}})=0$.%

(b) If $C_{k}$ and $C_{l}$ are components of $Z_{i,j}$ then, by Proposition
\ref{lem:RS_dif}(ii)(a), $\beta_{E}(\widetilde{\mathcal{M}})=1$. If $C_{k}$ and $C_{l}$ are components of $Z_{i,j}$ and $Z_{i,j}'$, respectively, then, by hypothesis, $E\subset\widetilde{C_{j}\times C_{k}}$ and according to Proposition \ref{lem:RS_dif}(ii)(c), we have $\beta_{E}(\widetilde{\mathcal{M}})=2$.

The cases (c) and (d) are similar and the proof is complete.
\end{proof}

\bigskip

\begin{corollary} \label{cor:tor}
The sheaf $\phi_* \widetilde{\mathcal{M}}$ is relatively torsion-free and rank-$1$ and its formation commutes with base changes.
\end{corollary}

\begin{proof}
By Proposition \ref{pro:admissible}, we need to show that $\widetilde{\mathcal{M}}$ is $\phi$-admissible and this follows from Propositions \ref{lem:RS_igual} and \ref{lem:RS_dif}.
\end{proof}


\section{Degree-1 Abel maps}


Recall some notation introduced in the previous sections. Let $\phi:\widetilde
{\mathcal{C}^{2}}\rightarrow\mathcal{C}^{2}$ be a good partial
desingularization of $\mathcal{C}^{2}$, $\mu:=\left.  \left(  p_{2}\circ
\phi\right)  \right\vert _{\widehat{C}}:\widehat{C}\rightarrow C,$ and $\rho=p_{1}\circ\phi$, where $p_{1}$ and $p_{2}$ are, respectively, the projections of $\mathcal{C}^{2}$ onto the first and second factors.

Let $R\in C_{i}\cap C_{j}$ be a fixed reduced node of $C$. We know that
$\widehat{C}=\rho^{-1}\left(  R\right)  =\mu^{-1}(C)$. Let $\Delta
\subset\mathcal{C}^{2}$ be the diagonal subscheme and let $\widetilde{\Delta}$
be the strict transform of $\Delta$ (via $\phi$).

We will prove that the sheaf $\phi_* \widetilde{\mathcal{M}}$ is $\sigma$-quasistable, where
\begin{equation}
\widetilde{\mathcal{M}}:=\mathcal{I}_{\widetilde{\Delta}/\widetilde{\mathcal{C}^{2}}%
}\otimes\left(  p_{2}\phi\right)  ^{\ast}\mathcal{L}\otimes\widetilde
{\mathcal{T}}, \label{eqn2}%
\end{equation}
is a sheaf over $\widetilde{\mathcal{C}^{2}}/\mathcal{C}$ with $\mathcal{L}$ is a line bundle over $\mathcal{C}/B$, $\widetilde{\mathcal{T}}=\bigotimes\limits_{i=1}^{p}\widetilde{T_{i}}$ and $T_{i}$ is the twister for each component $C_{i}$ of $C$.

Let $\sigma:\mathcal{C}\rightarrow\mathcal{C}^{2}$ be a section through the
smooth locus of $f:\mathcal{C} \rightarrow B$ such that $\sigma(0)=P$. Let $\varphi$ be the restriction of $\phi$ to the inverse image of the smooth locus of
$\mathcal{C}^{2}$. Since $\varphi$ is an isomorphism, there exists a lifting
of $\sigma$ to $\widetilde{\mathcal{C}^{2}}$, which we also denote by $\sigma
$. We will denote $\widehat{P}:=\phi^{-1}(P)$, which is a smooth point of $\widehat{C}$

A key tool to prove $\sigma$-quasistability of $\widetilde{\mathcal{M}}$ is the following result.

\begin{proposition} \label{pro:beta2}
Let $\psi:\mathcal{Y} \rightarrow \mathcal{C}^2$ be a good partial desingularization of $\mathcal{C}^2$. Let $\mathcal{L}$ and $\mathcal{M}$ be $\psi$-admissible invertible sheaves on $\mathcal{Y}$. Let $Y$ and $X$ be fibers, respectively of

\begin{displaymath}
\xymatrix{
\mathcal{Y} \ar[r]^{\psi} & \mathcal{C}^{2} \ar[r]^{p_i} & \mathcal{C} }
\end{displaymath} 
and of

\begin{displaymath}
\xymatrix{
\mathcal{C}^{2} \ar[r]^{p_i} & \mathcal{C} }
\end{displaymath} 
such that $Y=\psi^{-1}(X)$, where $p_i$ is the projection onto the $i$-th factor. Let $L$ and $M$ be, respectively, the restrictions of $\mathcal{L}$ and $\mathcal{M}$ to $Y$ and assume that
\[L \otimes M^{-1}=\Ocal_Y \left(\sum_i a_iE_i \right),\]
where the sum runs over all $E_i \subset Y$ contracted by $\psi$ and $a_i \in \mathbb{Z}$. 

Then, $\psi_*(\mathcal{L}) \cong \psi_*(\mathcal{M})$.
\end{proposition}

\begin{proof}
\cite[Proposition 3.2, p.67]{EP}
\end{proof}

\bigskip

Recall that, by Proposition \ref{pro:fiber}, $\widehat{C}=\rho^{-1}(R)$ consists of one of the following types: $\widehat{C} \cong C$ if $R$ is a smooth point; $\widehat{C} \cong C_R$ if $R$ is an irreducible node or $\widehat{C} \cong C(1)$ if $R$ is a reducible node of $C$. In this way, each connected subcurve $\widehat{Y}$ of
$\widehat{C}$ is of the form:%
\[
\widehat{Y}=A\cup B\cup D,
\]
where,

\begin{itemize}
\item $A=\bigcup_{k=1}^{r}\widehat{C}_{i_{k}},$ with $\{i_{1},\cdots,i_{r}\}\subset\{1,\cdots,p\}$ and $\mu\left(  \widehat{C}_{i_{k}}\right)  \cong C_{i_{k}}$ with $C_{i_{k}}$ an irreducible component of $C$;
\item $B=\bigcup E_{i_{l}},$ with $E_{i_{l}}$ a smooth rational component which is equal to $\mu^{-1}(R)$ for some node $R$ of $Y=\bigcup^r_{k=1}C_{i_k}$;
\item $D=\bigcup E_{i_{m}},$ with $E_{i_{m}}$ a smooth rational component which is equal to $\mu^{-1}(R)$ for some node $R \in Y\cap Y'$.
\end{itemize}

\bigskip

In this case, we say that $\widehat{Y}$ is a $Y$\textit{-lifting}. Note that each
subcurve $Y$ gives rise to more than one subcurve $\widehat{Y}$, however a given $\widehat{Y}$ is the $Y$-lifting of exactly one subcurve $Y$ of $C$.

Our goal is to prove that $\phi_* \widetilde{\mathcal{M}}$ is $\sigma$-quasistable. Since by Corollary \ref{cor:tor} the formation of $\phi_* \widetilde{\mathcal{M}}$ commutes with base change, it suffices to show that the sheaf $\phi_* \left(\left.\widetilde{\mathcal{M}} \right\vert _{\widehat{C}}\right)$ is $P$-quasistable for every fiber $\widehat{C}$ of $\rho: \widetilde{\mathcal{C}^2} \rightarrow \mathcal{C}$.

Fix a fiber $\widehat{C}$ of $\rho$. We define the sheaf
\[
\mathcal{G}:=\left.\widetilde{\mathcal{M}} \right\vert _{\widehat{C}} \left( \sum_i E_i \right),
\]
where the sum is taken over all $\phi$-exceptional component $E_i$ of $\widehat{C}$ such that $\beta_{E_i}(\widetilde{\mathcal{M}})=2$.

By definition of $\mathcal{G}$ and by Propositions \ref{lem:RS_igual} and \ref{lem:RS_dif}, we have $\beta_E(\widetilde{\mathcal{M}}) \in \{ 0,1,2 \}$ and hence $\beta_E(\mathcal{G}) \in \{ 0,1 \}$. Indeed, suppose that $\beta_E(\widetilde{\mathcal{M}})=2$. Since $\deg_E(\widetilde{\mathcal{E}})=0$ we have  $\beta_E(\mathcal{G})=\chi(\mathcal{G}|_E)=\deg_E(\mathcal{G}|_E) + 1=\deg_E(\widetilde{\mathcal{M}})-2+1=\beta_E(\widetilde{\mathcal{M}})-2=0$.

\begin{proposition} \label{pro:quasi}
The sheaf $\mathcal{G}$ is $\widehat{P}$-quasistable.
\end{proposition}

\begin{proof}
If $\widehat{C}$ is a fiber over a smooth point $R \in C_i$ then, by Proposition \ref{pro:fiber}, $\widehat{C} \cong C$. In this case $\mathcal{G} \cong \widetilde{\mathcal{M}}| _{C}$ and $\widetilde{\mathcal{T}} \cong \Ocal_{\widetilde{\mathcal{C}^2}}(\widetilde{C_i \times Z_{i}})$. Therefore,
\[
\mathcal{G} \cong \left.\widetilde{\mathcal{M}} \right\vert _{C} \cong m_R \otimes L \otimes T_i
\]
which is $\widehat{P}$-quasistable.

We can assume that $\widehat{C}$ is the fiber over a node of $C$. Let $\widehat{Y}$ a proper subcurve of $\widehat{C}$. We have to prove that $\beta_{\widehat{Y}}(\mathcal{G}) \geq 0$ and $\beta_{\widehat{Y}}(\mathcal{G}) > 0$ if $\widehat{P} \in \widehat{Y}$. For this we will compare $\beta_{\bullet}(\mathcal{G})$ with $\beta_{\bullet}({\widetilde{\mathcal{M}}}|_C)$ where $C$ is a fiber over a smooth point. In the rest of this proof we will denote $\beta_{\bullet}({\widetilde{\mathcal{M}}}|_C)$ by $\beta_{\bullet}(\widetilde{\mathcal{ M}})$. We have to consider two cases.

\bigskip
\textbf{(I)} $\widehat{C}$ is the fiber over an irreducible node $R$.

In this case, by Proposition \ref{pro:fiber}, there is only one $\phi$-exceptional component $E$ contained in $\widehat{C}$. Suppose $R \in C_i$. We have three possibilities for a subcurve $\widehat{Y} \subset \widehat{C}$.

If $\widehat{Y}=E$ then, since $R$ is an irreducible node, we have
\[ \deg_E(\widetilde{\mathcal{T}})=0 \text{ and } \deg_E(\widetilde{\mathcal{M}})=\deg_E(\mathcal{I}_{{\widetilde{\Delta}}/\widetilde{\mathcal{C}^2}})=1.\]
It follows that $\beta_{\widehat{Y}}(\widetilde{\mathcal{M}})=2$. By definition of the sheaf $\mathcal{G}$, we have $\beta_{\widehat{Y}}(\mathcal{G})=0$.

If $\widehat{Y}$ is a $Y$-lifting and $Y \not\supset C_i$ then, $\widehat{Y} \not\supset E$. Hence $\widehat{Y} \cong Y$ and
\[
\beta_{\widehat{Y}}(\mathcal{G})=\beta_{Y}(\widetilde{\mathcal{M}}) \geq 0.
\]

If $\widehat{Y}$ is a $Y$-lifting and $Y \supset C_i$ then, we have two subcases. First, if $E \subset \widehat{Y}$ then, by the flatness of $\rho$ we have $\beta_{\widehat{Y}}(\mathcal{G})=\beta_{Y}(\widetilde{\mathcal{M}}) \geq 0$. Second, if $E \not\subset \widehat{Y}$ then, the first subcase implies $\beta_{\widehat{Y} \cup E}(\mathcal{G}) \geq 0$. Since $\beta_E (\mathcal{G}) = 0$, it follows from Lemma \ref{lem:beta} that $\beta_{\widehat{Y}}(\mathcal{G})+\beta_E(\mathcal{G})-2 \geq 0$ and, therefore, $\beta_{\widehat{Y}}(\mathcal{G}) \geq 2$. 

\bigskip
\textbf{(II)} $\widehat{C}$ is the fiber over a reducible node $R \in C_i \cap C_j$.

In this case, by Proposition \ref{pro:fiber}, $\widehat{C} \cong C(1)$ where $C(1)$ is obtained from $C$ by replacing each reducible node by an $\phi$-exceptional component.

If $\widehat{Y}=E$ with $E$ a $\phi$-exceptional component then $\beta_E(\mathcal{G}) \in \{ 0,1 \}$.

If $\widehat{Y}$ contains components which are not $\phi$-exceptional then, let $Y$ be the subcurve of $C$ such that $\widehat{Y}$ is a $Y$-lifting. Either $Y=C$ or $Y$ is a proper subcurve of $C$.

If $Y=C$ then, we can write 
\[ \widehat{C}=\widehat{Y} \cup B_0 \cup B_1,\]
where $B_0,B_1$ are sets of $\phi$-exceptional components of $\widehat{C}$ such that $\beta_E (\mathcal{G})=0$ for every $E \in B_0$ and $\beta_E (\mathcal{G})=1$ for every $E \in B_1$. Notice that, since $Y=C$, by definition of the curve $\widehat{C}$, each $\phi$-exceptional component contained either $B_0$ or $B_1$ intersects $\widehat{Y}$ in two points. So, by Lemma \ref{lem:beta}, we have
\[
0=\beta_{\widehat{C}}(\mathcal{G})=\beta_{\widehat{Y}}(\mathcal{G})-2\#B_0 - \#B_1.
\]
Since $\widehat{Y}$ is proper, at least one of $B_0$ or $B_1$ is nonempty and so, $\beta_{\widehat{Y}}(\mathcal{G})>0$.

If $Y$ is a proper subcurve of $C$ then, we can write $Y=Y_1 \cup Y_2$, with $Y_1 \subset Z_{i,j}$ and $Y_2 \subset Z_{i,j}'$, where $Z_{i,j}$ is the twister difference subcurve with respect to $i$ and $j$ (see Definition \ref{def:subcurve}). Notice that $Y_1$ or $Y_2$ may be empty, that is, $Y$ may be contained in $Z_{i,j}$ or in $Z_{i,j}'$. We can write $\widehat{Y}=\widehat{Y_1} \cup \widehat{Y_2}$, where $\widehat{Y_k}$ is a $Y_k$-lifting for $k=1,2$. Let
\[
\widetilde{Y_1}:=\widetilde{C_i \times Y_1} \cap \widehat{C}
\]
and
\[
\widetilde{Y_2}:=\widetilde{C_j \times Y_2} \cap \widehat{C}.
\]
Notice that, by construction, $\widetilde{Y_1}$ and $\widetilde{Y_2}$ contains all the $\phi$-exceptional components of $\widehat{C}$ contained, respectively, in $\widetilde{C_i \times Y_1}$ and $\widetilde{C_j \times Y_2}$. By Proposition \ref{cor:betas} we see that each $\phi$-exceptional component $E$ contained in $\widetilde{Y_1}$ satisfies $\beta_E(\widetilde{\mathcal{M}}) \in \{0,1\}$. So, by the flatness of $\rho$ and by the definition of $\mathcal{G}$ we have
\begin{equation}
\beta_{Y_1}(\widetilde{\mathcal{M}})=\beta_{\widetilde{Y_1}}(\mathcal{G}). \label{eq1}
\end{equation}
Similarly, each $\phi$-exceptional component $E$ contained in $\widetilde{Y_2}$ satisfies $\beta_E(\widetilde{\mathcal{M}}) \in \{0,1\}$, hence
\begin{equation}
\beta_{Y_2}(\widetilde{\mathcal{M}})=\beta_{\widetilde{Y_2}}(\mathcal{G}). \label{eq2}
\end{equation}

We need to prove the following two claims.

\bigskip
\begin{claim}
Every $\phi$-exceptional component $E$ not contained in the region $\widetilde{C_i \times Y_1}$ and such that $E$ intersects $\widetilde{Y_1}$ satisfies 
\[ \beta_{\widetilde{Y_1} \cup E}(\mathcal{G})=\beta_{\widetilde{Y_1}}(\widetilde{\mathcal{M}}).\]
\end{claim}

Indeed, by the definition of the sheaf $\mathcal{G}$, we have $\beta_E(\mathcal{G}) \in \{ 0,1 \}$, for every $\phi$-exceptional component $E$. If $E$ is such that $\beta_E(\mathcal{G})=1$ then, by Lemma \ref{lem:beta} and by Equation (\ref{eq1}), we have
\[\beta_{\widetilde{Y_1} \cup E}(\mathcal{G})=\beta_{\widetilde{Y_1}}(\mathcal{G})+\beta_E(\mathcal{G}) - 1= \beta_{\widetilde{Y_1}}(\mathcal{G})=\beta_{\widetilde{Y_1}}(\widetilde{\mathcal{M}}).\]

If $E$ is such that $\beta_E(\mathcal{G})=0$ then, since by hypothesis $E \not\subset \widetilde{C_i \times Y_1}$ and $E$ intersects $\widetilde{Y_1}$, it follows that $E \subset \widetilde{C_j \times Y_1}$. By definition of the sheaf $\mathcal{G}$, we have $\beta_E(\widetilde{\mathcal{M}})=2$. So, again by Lemma \ref{lem:beta} and by Equation (\ref{eq1}), we have
\[\beta_{\widetilde{Y_1} \cup E}(\mathcal{G})=\beta_{\widetilde{Y_1}}(\mathcal{G}) - 1= \beta_{\widetilde{Y_1}}(\widetilde{\mathcal{M}})+1-1=\beta_{\widetilde{Y_1}}(\widetilde{\mathcal{M}}).\]

In a similar way we can prove that
\[ \beta_{\widetilde{Y_2} \cup E}(\mathcal{G})=\beta_{\widetilde{Y_2}}(\widetilde{\mathcal{M}}),\]
for every $\phi$-exceptional component $E$ not contained in $\widetilde{C_2 \times Y_2}$ and such that $E$ intersects $\widetilde{Y_2}$. $\hfill \square$

\bigskip
\begin{claim}
For every $Y_k$-lifting $\overline{Y}_k$ contained in $\widetilde{Y_k}$ and $k=1,2$, we have
\[ \beta_{\overline{Y}_k}(\mathcal{G}) \geq \beta_{\widetilde{Y_k}}(\mathcal{G}).\]
\end{claim}

Indeed, by the definition  of the sheaf $\mathcal{G}$, we have $\beta_E(\mathcal{G}) \in \{ 0,1 \}$, for every $\phi$-exceptional component $E$. If $\overline{Y}_i \subset \widetilde{Y_i}$ is a $Y_i$-lifting then,
\[\widetilde{Y_i}=\overline{Y}_i \cup A_1 \cup A_0,\]
where $A_1 \subset \widetilde{Y_k}$ is the set of $\phi$-exceptional components contained in $\widetilde{Y_k}$, not contained in $\overline{Y}_k$ and such that $\beta_E(\mathcal{G})=1$ for every $E \in A_1$ and $A_0 \subset \widetilde{Y_k}$ is the set of $\phi$-exceptional components contained in $\widetilde{Y_k}$, not contained in $\overline{Y}_k$ and such that $\beta_E(\mathcal{G})=0$ for every $E \in A_0$. By Lemma \ref{lem:beta}, we have
\begin{eqnarray*}
\beta_{\widetilde{Y_k}}(\mathcal{G}) & \leq & \beta_{\overline{Y}_k}(\mathcal{G}) + \beta_{A_1}(\mathcal{G}) - \#(A_1 \cap \overline{Y}_k) + \beta_{A_0}(\mathcal{G}) - \#(A_0 \cap \overline{Y}_k) \\
 & = & \beta_{\overline{Y}_k}(\mathcal{G}) + \#A_1 - \#(A_1 \cap \overline{Y}_k) - \#(A_0 \cap \overline{Y}_k),
\end{eqnarray*}
where the above inequality follows from the fact aht each $\phi$-exceptional component contained in either $A_1$ or $A_0$ intersects $\widehat{Y_k}$ in either $1$ or $2$ points. Hence, the integer
\[\#A_1 - \#(A_1 \cap \overline{Y}_k) - \#(A_0 \cap \overline{Y}_k)\]
in the above equation is nonpositive and we conclude that $\beta_{\overline{Y}_k}(\mathcal{G}) \geq \beta_{\widetilde{Y_k}}(\mathcal{G}). \hfill \square$

\bigskip
We can conclude the proof as follows. By Lemma \ref{lem:beta} and by the fact that the sheaf $\widetilde{\mathcal{M}}$ is generically $\sigma$-quasistable, we have
\[\beta_{Y_1 \cup Y_2}(\widetilde{\mathcal{M}})=\beta_{Y_1}(\widetilde{\mathcal{M}})+\beta_{Y_2}(\widetilde{\mathcal{M}})- \#(Y_1 \cap Y_2) \geq 0,\]
with strict inequality if $P \in Y_1 \cup Y_2$. By claims $1$ and $2$, to check that $\mathcal{G}$ is $\widetilde{\sigma}$-quasistable, we can reduce to check the condition for the case $\widehat{Y_k}=\widetilde{Y_k}$, $k=1,2$.

Let $D_0$ and $D_1$ be respectively the sets of $\phi$-exceptional components of $\widehat{C}$ contained in $\widetilde{Y_1} \cap \widetilde{Y_2}$ and such that $\beta_E(\mathcal{G})=0$ for every $E \in D_0$ and $\beta_E(\mathcal{G})=1$ for every $E \in D_1$. We can write
\[\widetilde{Y_1} = \overline{Y}_1 \cup D_0 \cup D_1 \;\text{ and }\; \widetilde{Y_2} = \overline{Y}_2 \cup D_0 \cup D_1,\]
where 
\begin{equation}
\overline{Y}_k=\widetilde{Y_k} \setminus (D_0 \cup D_1), \label{eq3}
\end{equation}
for $k=1,2$. By Lemma \ref{lem:beta} and by Equation (\ref{eq3}), we have
\begin{eqnarray}
\beta_{\widehat{Y}}(\mathcal{G}) & = & \beta_{\widetilde{Y_1} \cup \widetilde{Y_2}}(\mathcal{G}) \nonumber \\
& = & \beta_{\widetilde{Y_1} \cup \overline{Y}_2}(\mathcal{G}) \nonumber \\
& = & \beta_{\widetilde{Y_1}}(\mathcal{G}) + \beta_{\overline{Y}_2}(\mathcal{G}) - \#(\widetilde{Y_1} \cap \overline{Y}_2) \nonumber \\
& \geq & \beta_{\widetilde{Y_1}}(\mathcal{G}) + \beta_{\widetilde{Y_2}}(\mathcal{G}) - \#(\widetilde{Y_1} \cap \overline{Y}_2) \label{eq4}\\
& = & \beta_{\widetilde{Y_1}}(\widetilde{\mathcal{M}}) + \beta_{\widetilde{Y_2}}(\widetilde{\mathcal{M}}) - \#(\widetilde{Y_1} \cap \overline{Y}_2) \nonumber \\
& = & \beta_{Y_1}(\widetilde{\mathcal{M}}) + \beta_{Y_2}(\widetilde{\mathcal{M}}) - \#(\widetilde{Y_1} \cap \overline{Y}_2) \label{eq5}\\
& \geq & \beta_{Y_1}(\widetilde{\mathcal{M}}) + \beta_{Y_2}(\widetilde{\mathcal{M}}) - \#(Y_1 \cap Y_2) \label{eq6}\\
& \geq & 0 \label{eq7}
\end{eqnarray}
where Inequality (\ref{eq4}) follows from Claim 2, Equation (\ref{eq5}) follows from the flatness of the sheaf $\widetilde{\mathcal{M}}$ and Inequality (\ref{eq6}) follows from the fact that  $\#(Y_1 \cap Y_2) \geq \#D_0 + \#D_1$. Note that if $\widehat{P} \in \widehat{Y}$ then Inequality (\ref{eq7}) is strict.

The proof is complete.
\end{proof}

\bigskip

\begin{proposition} \label{pro:sigma}
The sheaf $\phi_{\ast}\widetilde{\mathcal{M}}$ is $\sigma$-quasistable.
\end{proposition}

\begin{proof}
By Propositions \ref{lem:RS_igual} and \ref{lem:RS_dif}, the sheaves $\mathcal{G}$ and $\left. \widetilde{\mathcal{M}} \right\vert _{\widehat{C}}$ are $\phi$-admissible. We know that $\mathcal{G}$ is $\widehat{P}$-quasistable, by Proposition \ref{pro:quasi}. So $\phi_{\ast}(\mathcal{G})$ is $\sigma$-quasistable, by Propositions \ref{pro:admissible} and \ref{pro:stability}. By Proposition \ref{pro:beta2}, $\phi_{\ast} \left(\left. \widetilde{\mathcal{M}} \right\vert _{\widehat{C}}\right) \cong \phi_{\ast}(\mathcal{G})$. Then, $\phi_{\ast} \left(\left. \widetilde{\mathcal{M}} \right\vert _{\widehat{C}}\right)$ is $\sigma$-quasistable. Finally, by Corollary \ref{cor:tor}, $\phi_{\ast} \widetilde{\mathcal{M}}$ is $\sigma$-quasistable.
\end{proof}

\bigskip

\begin{theorem}[Main result]\label{teo:main}
The sheaf $\phi_*\widetilde{\mathcal{M}}$ on $p_1: \mathcal{C}^2 \rightarrow \mathcal{C}$ induces a morphism
\[
\overline{\alpha}_{\mathcal{L},\mathcal{E}}:%
\begin{array}
[t]{ccl}%
\mathcal{C} & \rightarrow & \overline{J}_{\mathcal{E}}^{\sigma}\\
Q & \mapsto & \left[  \left.  \phi_{\ast}\widetilde{\mathcal{M}}\right\vert _{p_{1}%
^{-1}(Q)}\right]
\end{array}
\]
restricting to $\alpha_{\mathcal{L},\mathcal{E}}$ over the smooth locus of $f: \mathcal{C} \rightarrow B$.
\end{theorem}

\begin{proof}
By Corollary \ref{cor:tor} the s heaf $\phi_{\ast}\widetilde{\mathcal{M}}$ is torsion-free of rank-$1$. By Proposition \ref{pro:sigma} $\phi_{\ast}\widetilde{\mathcal{M}}$ is $\sigma$-quasistable and hence $\overline{\alpha}_{\mathcal{L},\mathcal{E}}$ is a morphism. As $\phi$ is an isomorphism away from the exceptional components, it follows that $\overline{\alpha}_{\mathcal{L},\mathcal{E}}$ extend $\alpha_{\mathcal{L},\mathcal{E}}$. Thus, the proof is complete.
\end{proof}

\bigskip

\noindent \textbf{Acknowledgements.} We would like to thank Marco Pacini, Juliana Coelho and Alex Abreu for introducing us to the subject and for helping us to prepare this article.

\pagebreak

\bibliographystyle{aalpha}

\bigskip

\begin{thebibliography}{99999}                                                                                             %
\bibitem[ACoP]{ACoP}A. ABREU, J. COELHO and M. PACINI, \textit{On the geometry
of Abel maps for nodal curves}, Michigan Math. J. \textbf{64}, N${{}^o}$ 1 (2015), 77-108.

\bibitem[AK76]{AK76}A. ALTMAN and S. KLEIMAN, \textit{Compactifying the
Jacobian}, Bull. Amer. Math. Soc. \textbf{6} (1976) 947-949.

\bibitem[AK80]{AK80}A. ALTMAN and S. KLEIMAN, \textit{Compactifying the Picard
scheme}, Adv. Math. \textbf{35} (1980), 50-112.

\bibitem[C]{C}L. CAPORASO, A compactification of the universal Picard variety
over the moduli space of stable curves, J. Amer. Math. Soc. \textbf{7} (1994), 589-660.

\bibitem[CCoE]{CCoE}L. CAPORASO, J. COELHO and E. ESTEVES, \textit{Abel maps
of Gorenstein curves, }Rend. Circ. Mat. Palermo (2) \textbf{57}, N${{}^o}$ 
1 (2008), 33-59.

\bibitem[CE]{CE}L. CAPORASO, E. ESTEVES, \textit{On Abel maps of stable
curves}. Mich. Math. J. \textbf{55} (2007) 575-607.

\bibitem[Co]{Co}J. COELHO, \textit{Abel maps for reducible curves}, Doctor
thesis, IMPA, 2007, available at $<$http://www.preprint.impa.br/Shadows/SERIE
C/2007/57.html$>$.


\bibitem[CoEP]{CoEP}J. COELHO, E. ESTEVES and M. PACINI, \textit{Degree-2 Abel maps for nodal curves}, Int. Math. Res. Not. IMRN \textbf{2016}, N${{}^o}$ 10 (2016), 2912-2973.

\bibitem[CoP]{CoP}J. COELHO and M. PACINI, \textit{Abel maps for curves of
compact type}, J. Pure Appl. Algebra \textbf{214}, N${{}^o}$ 
8 (2010), 1319-1333.

\bibitem[DS]{DS}C. D'SOUZA, \textit{Compactification of generalised Jacobians.}, Proc. Indian Acad. Sci. Sect. A Math. Sci. \textbf{88}, N${{}^o}$ 5 (1979), 419-457.


\bibitem[E01]{E01}E. ESTEVES, \textit{Compactifyng the relative Jacobian over
families of reduced curves}, Trans. of Amer. Math. Soc. \textbf{353} (2001), 3045-3095.

\bibitem[E09]{E09}E. ESTEVES, \textit{Compactified Jacobians of curves with spine decompositions}, Geom. Dedicata \textbf{139} (2009), 167-181.

\bibitem[EP]{EP}E. ESTEVES and M. PACINI, \textit{Semistable modifications of families of curves and compactified Jacobians}, Ark. Mat. \textbf{54}, N${{}^o}$ 1 (2016), 55-83.

\bibitem[EGA]{EGA}A. GROTHENDIECK and J. DIEUDONN\'{E},
\textit{\'{E}l\'{e}ments de G\'{e}om\'{e}trie Alg\'{e}brique}, Publ. Math.
IHES \textbf{24} (1965) and \textbf{28} (1966).


\bibitem[HM]{HM}J. HARRIS and I. MORRISON, \textit{Moduli of curves}.
Springer, New York, 1998.

\bibitem[Ha]{Ha}R. HARTSHORNE, \textit{Algebraic Geometry}. Springer, New
York, 1977.

\bibitem[I]{I}J. IGUSA, \textit{Fibre systems of Jacobian varieties.} Amer. J. Math. \textbf{78} (1956), 171-199.

\bibitem[N]{N}P. E. NEWSTEAD, \textit{Lectures on Introduction to Moduli
Problems and Orbit Spaces. }Tata Institute Lecture Notes. [Lectures on
Mathematics and Physics, v. 51]. Springer-Verlag, New York, 1978.

\bibitem[OS]{OS}T. ODA and C. S. SESHADRI, \textit{Compactifications of the generalized Jacobian variety.}, Trans. Amer. Math. Soc. \textbf{253} (1979), 1-90.

\bibitem[P1]{P1}M. PACINI, \textit{The resolution of the degree-2 Abel-Jacobi map for nodal curves-I}, Math. Nachr. \textbf{287}, N${{}^o}$ 17-18 (2014), 2071-2101.

\bibitem[P2]{P2}M. PACINI, \textit{The resolution of the degree-2 Abel-Jacobi map for nodal curves-II} (preprint), available at $<$https://arxiv.org/abs/1304.5288$>
$.

\bibitem[PP]{PP}R. PANDHARIPANDE, \textit{A compactification over $\overline{M}_g$ of the universal moduli space of slope-semistable vector bundles.}, J. Amer. Math. Soc. \textbf{9}, N${{}^o}$ 2 (1996), 425-471.

\bibitem[R]{R}F. F. ROCHA, \textit{On autoduality for tree like curves, Abel maps for stable curves and translations for the compactified Jacobians}, Doctor thesis, IMPA, Rio de Janeiro, 2013. 

\bibitem[Se]{Se} F. SERCIO, \textit{On the degree-1 Abel map for nodal curves and gonality of stable curves}, Doctor thesis, UFF, Niter\'{o}i, 2014.

\bibitem[S]{S}C. S. SESHADRI, \textit{Fibr\'{e}s vectoriels sur les courbes
alg\'{e}briques}, Ast\'{e}risque \textbf{96 }(1982).


\bibitem[Si]{Si}C. T. SIMPSON, \textit{Moduli of representations of the fundamental group of a smooth projective variety. I.}, Inst. Hautes Études Sci. Publ. Math. \textbf{79} (1994), 47-129.

\bibitem[So]{So}A. N. de SOUZA, \textit{On Abel maps and their fibers in low degrees}, Doctor thesis, UFF, Niter\'{o}i, 2014.

\end{thebibliography}

\end{document}